\documentclass[psamsfonts, a4paper,11pt]{article}

\usepackage[top=1in,right=1in,left=1in,bottom=1in,a4paper]{geometry}

\usepackage{amsmath, amsthm, amssymb, amsfonts}

\usepackage{enumerate}
\usepackage[all]{xy}

\newdir{ >}{!/-8.5pt/@{>}}

\newtheoremstyle{theoremstyle}
  {10pt}      
  {5pt}       
  {\itshape}  
  {}          
  {\bfseries} 
  {:}         
  {.5em}      
  {}          

\newtheoremstyle{examplestyle}
  {10pt}      
  {5pt}       
  {}          
  {}          
  {\bfseries} 
  {:}         
  {.5em}      
  {}          

\theoremstyle{theoremstyle}
\newtheorem{theorem}{Theorem}[section]
\newtheorem*{theorem*}{Theorem}
\newtheorem{lemma}[theorem]{Lemma}
\newtheorem*{lemma*}{Lemma}
\newtheorem{proposition}[theorem]{Proposition}
\newtheorem*{proposition*}{Proposition}
\newtheorem{corollary}[theorem]{Corollary}
\newtheorem*{corollary*}{Corollary}
\newtheorem{convention}[theorem]{Convention}

\theoremstyle{examplestyle}
\newtheorem{example}[theorem]{Example}

\newtheorem{definition}[theorem]{Definition}
\newtheorem{definition*}{Definition}
\newtheorem{remark}[theorem]{Remark}
\newtheorem{remark*}{Remark}

\newcommand{\comment}[1]{}

\newcommand{\sh}[1]{\mathcal{#1}}
\newcommand{\msh}[1]{$\sh{#1}$}
\newcommand{\spec}[1]{\operatorname{spec}(#1)}

\newcommand{\dual}[1]{{\check{#1}}}

\newcommand{\orb}[1]{\operatorname{orb}(#1)}

\newcommand{\rk}{\operatorname{rk}}
\newcommand{\ob}{\operatorname{Ob}}
\newcommand{\mor}{\operatorname{Mor}}
\newcommand{\Hom}{\operatorname{Hom}}

\newcommand{\ksm}{{\K[\sigma_M]}}

\newcommand{\ilim}{\varprojlim}
\newcommand{\colim}{\varinjlim}
\newcommand{\kernel}{\operatorname{ker}}

\newcommand{\lcm}{\operatorname{lcm}}
\newcommand{\op}{{\operatorname{op}}}
\newcommand{\zip}{\operatorname{zip}}
\newcommand{\lcmzip}{\operatorname{zip}_{\lcm}}
\newcommand{\gcdzip}{\operatorname{zip}_{\gcd}}
\newcommand{\unzip}{\operatorname{unzip}}
\newcommand{\lcmunzip}{\operatorname{unzip}_{\lcm}}
\newcommand{\gcdunzip}{\operatorname{unzip}_{\gcd}}

\newcommand{\A}{\mathbf{A}}
\newcommand{\E}{\mathbf{E}}
\newcommand{\K}{\mathbb{K}}
\newcommand{\Z}{\mathbb{Z}}

\newcommand{\R}{\mathbb{R}}
\newcommand{\N}{\mathbb{N}}

\newcommand{\p}{\mathcal{P}}

\newcommand{\cL}{\mathcal{L}}
\newcommand{\cG}{\mathcal{G}}
\newcommand{\mH}{\mathcal{H}}
\newcommand{\n}{{\underline{n}}}
\newcommand{\uc}{{\underline{c}}}
\newcommand{\ud}{{\underline{d}}}
\newcommand{\on}{{[n]}}

\newcommand{\Zbar}{\bar{\mathbb{Z}}}
\newcommand{\msim}{\mathord{\sim}}
\newcommand{\one}{\mathbf{1}}

\newcommand{\zsMod}{\text{$\Z^n$-$S$-$\operatorname{Mod}$}}
\newcommand{\zsModf}{\text{$\Z^n$-$S$-$\operatorname{Mod}_f$}}
\newcommand{\zsmod}{\text{$\Z^n$-$S$-$\operatorname{mod}$}}
\newcommand{\mksmMod}{\text{$M$-$\ksm$-$\operatorname{Mod}$}}
\newcommand{\mksmModf}{\text{$M$-$\ksm$-$\operatorname{Mod}_f$}}
\newcommand{\mksmmod}{\text{$M$-$\ksm$-$\operatorname{mod}$}}

\title{Resolutions and Cohomologies of Toric Sheaves.\\
The affine case}

\date{June 2011}

\author{Markus Perling\footnote{Fakult\"at f\"ur Mathematik Ruhr-Universit\"at Bochum
Universit\"atsstra\ss e 150 44780 Bochum, Germany, {\tt Markus.Perling@rub.de}}}

\begin{document}

\maketitle

\begin{center}
\begin{small}
\emph{Dedicated to G\"unther Trautmann on the occasion of his 70th birthday.}
\end{small}
\end{center}

\

\begin{abstract}
We study equivariant resolutions and local cohomologies of toric sheaves for affine toric
varieties, where our focus is on the construction of new examples of indecomposable
maximal Cohen-Macaulay modules of higher rank. A result of Klyachko states that the
category of reflexive toric sheaves is equivalent to the category of vector spaces
together with a certain family of filtrations. Within this setting, we develop machinery
which facilitates the construction of minimal free resolutions for the smooth case as
well as resolutions which are acyclic with respect to local cohomology functors for the
general case. We give two main applications. First, over the polynomial ring, we determine
in explicit combinatorial terms the $\Z^n$-graded Betti numbers and local cohomology of
reflexive modules whose associated filtrations form a hyperplane arrangement. Second,
for the non-smooth, simplicial case in dimension $d \geq 3$, we construct new examples of
indecomposable maximal Cohen-Macaulay modules of rank $d - 1$.
\end{abstract}

\tableofcontents

\section{Introduction}

In this article we want to study equivariant resolutions and local cohomologies of toric
sheaves for affine toric varieties. Our aim is to identify good invariants and to set up
the machinery which allows to determine explicitly Betti and Bass numbers and local
cohomologies in many interesting cases, as well as to construct new examples of
indecomposable maximal Cohen-Macaulay modules of rank greater than one.
Let $M$ and $N$ denote the character and co-character group of a $d$-dimensional
algebraic torus over some field
$\K$, and $\sigma_M \subseteq M$ a subsemigroup corresponding to a strictly convex rational
polyhedral cone $\sigma \subseteq N \otimes_\Z \R$ such that $\ksm$ is the coordinate ring of a normal affine toric variety $U_\sigma = \spec{\ksm}$. Equivariant sheaves on $U_\sigma$
then are equivalent to $M$-graded $\ksm$-modules.
Our main goal is to study $M$-graded local cohomologies of such modules
in terms of $M$-graded acyclic resolutions. For this, we proceed in two steps, where
we start with the case of free and injective $\Z^n$-graded resolutions over the polynomial
ring $S = \K[x_1, \dots, x_n]$ (our results are related to earlier work on injective
resolutions such as \cite{Miller00}, \cite{helmmiller2}, \cite{helmmiller}).
For the second step, we can assume by a construction of Cox that $\ksm$ is a subring of
some $S$,
where $n$ denotes the number of one-dimensional cones in $\sigma$. More precisely,
we can identify $\ksm$ with the degree-zero part of $S$ with respect to its grading by
a certain abelian group $A$ (see subsection \ref{homogeneouscoordinates}). The $A$-grading
on $S$ is compatible with its natural $\Z^n$-grading in the sense that by taking degree
zero with respect to the $A$-grading, we obtain an essentially surjective functor from
the category of $\Z^n$-graded $S$-modules to the category $M$-graded $\ksm$-modules
(see \cite{CoxBat} Prop. 4.17). In particular, $\Z^n$-graded injective resolutions
descend to $M$-graded resolutions which are acyclic with respect to local cohomology with
monomial support.

The subject of $M$-graded $\ksm$-modules has been well studied. See the book
\cite{SturmfelsMiller} for a general overview on the subject and the articles
\cite{Roemer01}, \cite{Yanagawa01}, \cite{Yanagawa03}, \cite{Tchernev07} for
works which share
some common features with ours. The new aspect we bring into this subject is the
following. It was observed by Klyachko \cite{Kly90}, \cite{Kly91} that the class
of $M$-graded $\ksm$-modules is a natural extension of toric geometry in terms of
linear algebra. Therefore these modules exhibit a rich combinatorial content, as
is made explicit by the following theorem for the case of reflexive modules.

\begin{theorem}[\cite{Kly90}, \cite{Kly91}, see also \cite{perling1}]\label{klythm}
The category of finitely-generated, $M$-graded reflexive $\ksm$-modules is equivalent
to the category of vector spaces $\E$ endowed with $n$ filtrations
$0 \subseteq \cdots \subseteq E^k(i) \subseteq E^k(i + 1) \subseteq \cdots \subseteq \E$
for $k \in \on$ which are full in the sense that $E^k(i) = 0$ for $i << 0$ and $E^k(i)
= \E$ for $i >> 0$.
\end{theorem}

Note that here the set $\on = \{1, \dots, n\}$ is identified with the set of rays of the
cone $\sigma$ (see also our notations and conventions below). For the case of
$\Z^n$-graded $S$-modules, a standard method to extract combinatorial content from such
a module is to consider certain finite, adapted subsets of $\Z^n$, so-called
$\lcm$-lattices. Here, we consider $\Z^n$ as a poset by setting $(c_1, \dots, c_n) \leq
(c_1', \dots, c_n')$ iff $c_k \leq c_k'$ for every $k \in \on$. The $\lcm$ of any two
elements $\uc, \uc' \in \Z^n$ is defined by taking the componentwise maximum. Originally,
$\lcm$-lattices have been introduced
in \cite{GPW99} for monomial ideals. A definition for $\lcm$-lattices of general
$\Z^n$-graded $S$-modules has been proposed in \cite{CharalambousTchernev03}. In
general, one could consider the defining properties for an (admissible) $\lcm$-lattice
of a graded module $E$ that it is closed under taking $\lcm$s in $\Z^n$ and that it
contains all possible degrees of nonzero graded Betti numbers of $E$.
Theorem \ref{klythm} connects to this by giving rise to a very nice interpretation for
the $\lcm$-lattice of a reflexive module. Namely, let $E$ be such a module
with associated filtrations as in Theorem \ref{klythm} and and denote $\mathcal{V}$
the vector space arrangement in $\mathbf{E}$ generated by the intersections of
the vector spaces $E^k(i)$. Then we observe that the mapping $\mathcal{V} \rightarrow
\Z^n$ given by
\begin{equation*}
X \mapsto \underline{i}^X = (i_1^X, \dots, i_n^X),
\end{equation*}
where $i_k^X = \min \{i \mid X \subseteq E^k(i)\}$ embeds $\mathcal{V}$ as a poset
into $\Z^n$. The image of this map then indeed is a minimal admissible $\lcm$-lattice for
$E$ (see Proposition \ref{reflexivelcmlattice}). Now, let a minimal free graded free resolution
$0 \rightarrow F_t \rightarrow \cdots \rightarrow F_0 \rightarrow E \rightarrow 0$
be given, then the
modules $F_i$ as well as the syzygy modules are reflexive as well and to each of these we
have associated filtrations. Now, the most important structural
observation is that we can transport the notion of free resolutions over to the setting
of vector spaces arrangements. Forgetting about the modules and only considering the
induced maps of the underlying vector space arrangements, we obtain an exact sequence
of vector space arrangements:
\begin{equation*}
0 \longrightarrow \mathcal{F}_t \longrightarrow \cdots \longrightarrow \mathcal{F}_0
\longrightarrow \mathcal{V} \longrightarrow 0.
\end{equation*}
Here, the $\mathcal{F}_i$ denotes the coordinate arrangements associated to the free
modules $F_i$.
Cutting this sequence in short exact pieces $0 \rightarrow \mathcal{V}_{i + 1} \rightarrow
\mathcal{F}_i \rightarrow \mathcal{V}_i \rightarrow 0$ (with $\mathcal{V} = \mathcal{V}_0$),
we obtain an iterative construction of $\mathcal{V}_i$ as the $i$-th {\em syzygy arrangement}
of $\mathcal{V}$. As we will see in subsection \ref{reflres}, the construction of syzygy arrangements
can be done intrinsically in the category of vector space arrangements. Indeed, for reflexive
modules with equivalent underlying vector space arrangements, we obtain equivalent
free resolutions in the sense that we can identify their nonzero graded Betti numbers
via a poset isomorphism between the sets of nonzero graded Betti numbers.

\begin{theorem*}[\ref{invariants}]
Let $E$ be a $\Z^n$-graded, finitely generated, reflexive $S$-module. Then the poset of
nonzero graded Betti numbers is determined by the embedding of the poset given by the
underlying vector space arrangement $\mathcal{V}$ into $\Z^n$
For given $X \in \mathcal{V}$, the corresponding Betti number depends only on
$\mathcal{V}$.
\end{theorem*}

The remarkable implication of this theorem is that by our approach we obtain new invariants
of $\Z^n$-graded modules.
Most of this paper will be devoted to setting up machinery to make above construction work
and to utilize it for local cohomology computations both over $S$ and over $\ksm$. This will
in particular be done by adapting results of Miller \cite{Miller00} to our setting (see subsections
\ref{duality} and \ref{loccohomcomp}).

To show the versatility of our methods we will present two important applications. The
first will
be the explicit computation of graded Betti numbers and local cohomology of reflexive
$S$-modules whose associated filtrations form hyperplane arrangements. Assume that
$E$ is a reflexive $\Z^n$-graded $S$-module whose associated filtrations form an
essential hyperplane arrangement $\mH$. We denote $X \mapsto \uc^X$ the
poset embedding of $\mH$ into $\Z^n$. Then the graded Betti numbers are given by
$\beta_i(\uc^X)$ for $i \geq 0$ and $X \in \mH$. We get:

\begin{theorem*}[\ref{hyperplanebettinumbers}]
Let $E$ be a reflexive $\Z^n$-graded $S$-module whose associated filtrations form an
essential hyperplane arrangement $\mH$ .
Then for any $X \in \mH$ the associated graded Betti numbers $\beta_i(\uc^X)$
are zero unless $\dim X = i + 1$. If $\dim X = i + 1$, then $\beta_i(\uc^X)$
coincides with the beta invariant of $X$.
\end{theorem*}

Note that the beta invariant is a combinatorial invariant of a ranked poset which
can completely be expressed in terms of its M\"obius function (see subsection \ref{combinatorics}
for details). For a given $X \in \mH$, the beta invariant is the one associated to the
hyperplane arrangement $\{X \cap Y \mid Y \in \mH\}$ in $X$.

Now consider the local cohomology $H^i_x E$ of such modules with respect to the maximal
homogeneous ideal of $S$ (equivalently, with respect to the fixed point $x = 0$ of the
standard torus action on the affine space $\A^n_\K$). Given the degrees $\uc^X$ as above,
then by duality (see subsections \ref{duality} and \ref{loccohomcomp}), the graded pieces
of $H^i_x E$ are determined by the {\em $\gcd$-lattice} $\mathcal{G}$ (see subsection
\ref{lcmlattices}) generated by the Bass numbers of $E$, which are given by elements
$\{\ud^X \mid X \in \mH\}$, where $\ud^X = \uc^X - (1, \dots, 1)$ for every $X \in \mH$.
An element $\uc \in \Z^n$ is {\em adjacent} to $\ud \in \Z^n$ if $(H^i_x E)_\uc$ is
determined by (and therefore coincides with) $(H^i_x E)_\ud$. The following result shows
that the local cohomology of $E$ encodes the whole spectrum of beta invariants
of subarrangements of $\mH$.

\begin{theorem*}[\ref{basstheorem}]
Let $E$ be  a reflexive $\Z^n$-graded $S$-module whose associated filtrations form an
essential hyperplane arrangement $\mH$ in $V$.
Denote $\cG$ the $\gcd$-lattice generated by the degrees of the Bass numbers of
$E$. For every $\ud \in \cG$ denote $\mH^\ud \subset \mH$ the hyperplane
arrangement generated by those hyperplanes $H \in \mH$ with $\ud \nleq \ud^H$,
$r^\ud$ its rank and $\beta^\ud$ its beta invariant.
Then for any $\ud \in \cG$ and any $\uc \in \Z^n$ adjacent to $\ud$ we have
\begin{equation*}
\dim (H_x^i E)_\uc =
\begin{cases}
\beta^\ud & \text{ if \ } i = n - r^\ud + 1\\
0 & \text{ else}.
\end{cases}
\end{equation*}
\end{theorem*}

Our second application will be the explicit construction of new examples of $M$-graded
maximal Cohen-Macaulay (MCM) modules over $\ksm$ for the case that $\sigma$ is simplicial but
not regular. It is well known that there is --- up to degree-shift --- a one-to-one correspondence
between MCM modules of rank one and Weil divisor classes on $\spec{\ksm}$.
In \cite{AuslanderReiten89} Auslander and Reiten show that the toric ring
$\K[[X, Y, X]]^{\Z / 2\Z}$ is MCM-finite and that indeed there exists an MCM module
of rank two over this ring (see also \cite{Yoshino90} \S 16).
The main result of section \ref{mcmmodules} will be that this example fits into a general
class of examples of MCM modules over simplicial toric rings.

\begin{theorem*}[\ref{oneelement}]
Let $\sigma$ be a simplicial and non-regular cone of dimension $d \geq 3$.
Then there exists an
indecomposable $M$-graded MCM module of rank $d - 1$.
\end{theorem*}

In general, there is even more than one isomorphism class. More precise statements
will be given in subsection \ref{intersectionfreemcms}.

As a final remark I want to mention that in a late stage of writing this article
Bernd Sturmfels brought to my attention that some very similar ideas have been
developed independently in the context of topological data analysis. The theory
of multi-dimensional persistence essentially parallels the idea of
$\sigma$-families in \cite{perling1} from a topological point of view
(e.g. see \cite{CarlssonZomorodian09}, \cite{Knudson08}). It would be interesting
to see whether our work might lead to interesting applications in this area.

\subsection*{Acknowledgements.}

For conversations which brought new ideas into this project
I want to thank G\"unther Trautmann, Henning Krause, Ezra Miller, and Vic Reiner.

\subsection*{Overview of the paper.}

Section \ref{sectiontwo} contains general facts about poset representations and
graded modules; most of the material is well-known. In section \ref{resolutionsection}
we will develop our machinery for computing resolutions and local cohomologies of
graded modules. The results will be applied in section \ref{hyperplanearrangements} to
compute Betti and Bass numbers and the local cohomology of $\Z$-graded reflexive
$S$-modules whose associated filtrations generate hyperplane arrangements. Finally,
in section \ref{mcmmodules},
we construct examples of indecomposable MCM-modules of higher rank for the simplicial case.

\subsection*{Notation and general conventions.}

Throughout this text $\K$ shall denote a fixed field. The general setting of this
paper is that of finitely generated, normal monoid
rings over $\K$, i.e. rings of the form $\ksm$, where $\sigma_M = \check{\sigma} \cap M$.
Here, $M \cong \Z^d$ is identified with the character group of the $d$-dimensional
torus $T = \mathbb{G}_m^d(\K)$ and $N$ the dual group of cocharacters. Moreover,
$\sigma \subset N_\R = N \otimes_\Z \R$ denotes a strictly convex rational polyhedral
cone and $\check{\sigma} \subset M_\R = M \otimes_\Z \R$ its dual cone with respect to
the standard bilinear pairing $M \times N \rightarrow \Z$, $(m, n) \mapsto m(n)$.
Identifying $N$ with twice its dual, we will usually write $n(m)$ rather than $m(n)$.
Throughout we will assume that a $d$-dimensional cone $\sigma$ will be fixed.
Although we do not assume that $\K$ is algebraically closed, we will make
use of well-known standard facts of toric geometry, such as the correspondence
between the the faces of $\sigma$, $T$-orbits
on the associated affine toric $\K$-scheme $U_\sigma = \spec{\ksm}$ and $M$-graded
prime ideals in $\ksm$. Within the scope of this paper, the relevant results contained
in standard references such as \cite{Oda} and \cite{Fulton} are applicable to our
setting. For some $n \in \N$ we denote by $\on$ the ordered set $\{1 < \cdots < n\}$.
We will assume that $\sigma$ has $n$ rays with primitive
vectors $l_1, \dots, l_n$. By abuse of notation, we will in general not distinguish between
the $l_i$ and the rays they generate. Moreover, we will often identify the set
$\{l_1, \dots, l_n\}$ with $\on$.

We will throughout consider graded modules over graded commutative rings.
For general reference for general categorial properties of graded rings and modules
we refer to \cite{NO2}. Tensor products and $\Hom$ of graded modules will always be considered as graded. In particular, for any commutative ring $R$ graded by some (additive) abelian group $G$ and
some $G$-graded $R$-modules $E, F$, the module $\Hom_R(E, F)$ will be defined as
\begin{equation*}
\Hom_R(E, F) := \bigoplus_{g \in G} \Hom_R(E, F)_g
\end{equation*}
with $\Hom_R(E, F)_g = \Hom_R(E(-g), F)_0 = \Hom_R(E, F(g))_0$. We will also write
$\Hom^G_R(E, F)$ for $\Hom_R(E, F)_0$.
The graded tensor product of $E \otimes_R F$ is considered as graded by setting
$(E \otimes_R F)_g$ the submodule generated by elements $e \otimes f$ with
$e \in E_h$ and $f \in F_{g - h}$ for any $h \in G$. Alternatively, if $R$ is
a $\K$-algebra, $E \otimes F$ can be considered as the quotient of the
$\K$-vector space $E \otimes_\K F$, which is graded by setting $(E \otimes_\K F)_g
= \bigoplus_{h \in G} (E_h \otimes_\K F_{g - h})$, by the subvector space generated
by $re \otimes f - e \otimes rf$ for $e \in E, f \in F, r \in R$. Note that
$E(g) \otimes_R F \cong E \otimes_R F(g) \cong (E \otimes F)(g)$.
We will denote
$G$-$R$-Mod the category of $G$-graded $R$-modules where the morphisms are given
by $\Hom^G_R(E, F)$ for any two $G$-graded modules $E, F$.

\section{Graded modules and poset representations}\label{sectiontwo}

We will start in subsection \ref{posheaves} with some general remarks on sheaves
on posets and preordered sets, respectively. In subsection \ref{sigmafamilies} we
recapitulate material from \cite{perling1} on $\sigma$-families and add some
complementary remarks. In subsections \ref{matlissection} and \ref{minimizing}
we introduce Matlis duality and injective and and projective objects and discuss
resolutions. In subsection \ref{homogeneouscoordinates} we introduce divisorial
and codivisorial modules and resolutions via the homogeneous coordinate ring.

\subsection{Preliminaries on poset representations}
\label{posheaves}

In this work we will make extensive use of $\K$-linear representations of posets as
well as limits and colimits of such representations. Therefore it will be useful to
have several equivalent formulations for these kind of objects at hand. There is an
extensive literature about poset representations and their limits of which only a
very small part is relevant for us. For a recent survey we refer to \cite{Webb07}.
Basic information about limits and colimits can be found e.g. in
\cite{Eisenbud}, Appendix A6.

Let $\p$ be any preordered set with order relation $\leq$.
Recall that a preorder is defined by the same axioms as a partial order, except for the
reflexitivity
axiom, i.e. there may exist elements $x, y \in \p$ such that $x \leq y$ and $y \leq x$,
but $x \neq y$.
Then $\p$ in a natural way forms a category: its
objects are given by the set underlying $\p$ and the morphisms for $x, y \in \ob(\p)$ are:
\begin{equation*}
\mor(x, y) =
\begin{cases}
\text{the pair } (x, y) & \text{ if } x \leq y \\
\emptyset & \text{ else},
\end{cases}
\end{equation*}
with composition given by $(y, z) (x, y) = (x, z)$ whenever $x \leq y \leq z$.
Therefore the pair $(x, x)$ represents the identity morphism for all $x \in \mathcal{P}$.

\begin{definition}
A $\K$-linear representation of $\p$ is a functor from $\mathcal{P}$ to the category
of $\K$-vector spaces.
\end{definition}

Any such functor $E$ associates to $x \in \p$ a $\K$-vector space $E_x$ and to a
morphism $x \leq y$ in $\p$ a $\K$-linear homomorphism $E(x, y) : E_x \rightarrow E_y$.
It is straightforward to see that the $\K$-linear representations of $\p$ together
with their natural transformations form an abelian category.

We can always pass from the preordered set $\p$ to its canonically associated poset
$\p / \msim$, where $\sim$ denotes the equivalence relation on $\p$ given by $x \sim y$
iff $x \leq y$ and $y \leq x$. We get:

\begin{lemma}\label{popreequivrep}
The category of $\K$-linear representations of $\p$ is
equivalent to the category of $\K$-linear representations of $\p / \msim$.
\end{lemma}

\begin{proof}
Let $E$ be a $\K$-linear representation of $\p$. For any $x \sim y$, the morphism $(x, y)$
is an isomorphism in $\p$ and thus
$E(x, y)$ is an isomorphism of $\K$-vector spaces whose inverse is $E(y, x)$.
In particular, we have $E_x \cong E_y$ for every pair $x \sim y$.
To define a representation of $\p / \msim$, we set $E_{[x]} := \colim E_y$,
where $[x]$ denotes an equivalence class of $\sim$ and the colimit is taken over
all elements $y$ in this equivalence class. Then by the naturality of colimits,
we obtain a morphism $E([x], [y]): E_{[x]} \rightarrow E_{[y]}$ for any pair $(x, y)$.
These morphisms are compatible with composition in $\p$ by the functoriality of
colimits.

For the other direction, we can lift representation $E$ of $\p / \msim$ to a
representation of $\p$ by setting $E_x := E_{[x]}$ and $E(x, y) := E([x], [y])$.
It is straightforward to see that these two functors indeed establish an equivalence
of categories.
\end{proof}

\begin{remark}
The limit $\colim E_y$ in the proof of Lemma \ref{popreequivrep}
should be considered as a ``generic'' representative for the vector spaces $E_y$ with
$y$ in one given
equivalence class which does not depend on some particular choices. In particular,
$\colim E_y$ is isomorphic to any $E_y$.
\end{remark}

\begin{remark}
So, strictly speaking, it does not seem necessary to consider preordered sets
rather than just posets. However, as we will see later on, from the point of view
of toric geometry it will be more natural to first consider preordered sets.
\end{remark}

On $\p$ there is defined a topology which is generated by the basis
\begin{equation*}
U(x) := \{y \geq x\}
\end{equation*}
for all $x \in \mathcal{P}$. Note that the continuous maps between preordered sets then
are precisely the order preserving maps.

\begin{proposition}\label{repsheafequiv}
Let $\p$ be a preordered set. Then there is an equivalence of categories between the
categories of representations of $\p$ and of sheaves of $\K$-vector spaces
on $\p$.
\end{proposition}

\begin{proof}
A sheaf of $\K$-vector spaces \msh{E} on $\mathcal{P}$ with respect to this topology
automatically induces a $\K$-linear representation of $\mathcal{P}$ by setting
$E_x = \sh{E}\big(U(x)\big)$ for every $x \in \p$ and $E(x, y)$ the restriction
morphism $\sh{E}\big(U(x)\big) \rightarrow \sh{E}\big(U(y)\big)$. On the other hand,
for any representation $E$, following \cite{EGAI} \S 0.3.2,
we obtain a presheaf \msh{E} on $\mathcal{P}$ by setting $\sh{E}\big((U(x)\big)
:= E_x$ for all $x \in \mathcal{P}$ and $\sh{E}(U) := \ilim \sh{E}\big(U(x)\big)$
for some open set $U$, where
the limit runs over all $x \in U$. Note that the stalk $\sh{E}_x$ is isomorphic
to $\sh{E}\big(U(x)\big)$. Observing that for some $U(x)$ every open cover
of $U(x)$ necessarily contains $U(x)$ itself, we can apply the criterion of
\S 0.3.2.2 in \cite{EGAI} from which it follows that our presheaf is a sheaf.
\end{proof}

\subsection{$\sigma$-families}
\label{sigmafamilies}

In this subsection we will recall some material from \cite{perling1}. In \cite{perling1},
the general assumption was used that $\K$ is algebraically closed. However, the results
relevant for us actually do not depend on any properties of $\K$ and therefore will be
stated without any assumptions on $\K$.

We will fix some more notation. Elements of $M$ are denoted $m, m'$ etc. if written
additively and $\chi(m), \chi(m')$ etc. if written multiplicatively, i.e. $\chi(m + m')
= \chi(m)\chi(m')$. Faces of $\sigma$ are denoted by small Greek letters $\rho, \tau$
etc; the face order among faces is denoted $\rho \preceq \tau$. For any $\tau \preceq
\sigma$ we denote $\tau^\bot = \{m \in M_\R \mid \langle m, n \rangle \geq 0$ for all
$n \in \sigma\}$ and $\tau^\bot_M = \tau^\bot \cap M$; note that $\tau^\bot_M$ is the maximal subgroup of $\tau_M$. For $\tau \preceq \sigma$ we will denote $\tau'$
for $\tau$ considered as maximal cone in its $\R$-linear span in $N_\R$; then $\tau'_M$
equals $\sigma_M \cap \tau_M^\bot$. The associated monoid rings correspond
to the orbit $\orb{\tau} = \spec{\K[\tau_M^\bot]}$ associated to $\tau$ and its closure
$V(\tau) = \spec{\K[\tau'_M]}$ in $U_\sigma$. Moreover, $U_\tau$ splits into a product
$U_\tau \cong \orb{\tau} \times U_{\bar{\tau}}$, where $U_{\bar{\tau}} =
\spec{\K[\bar{\tau}_M]}$ and $\bar{\tau}_M$ is the image of $\tau_M$ under the projection
$M \twoheadrightarrow M / \tau_M^\bot$.

The notion of a $\sigma$-family is a simple reformulation of the notion of $M$-graded
$\ksm$-modules which will help us to make the combinatorial content of such modules more
explicit. The basic observation is that, given an $M$-graded $\ksm$-module
\begin{equation*}
E \cong \bigoplus_{m \in M} E_m,
\end{equation*}
its module structure is completely determined by the linear maps among the $E_m$ which are
given by multiplication with monomials, i.e. for any $m \in \sigma_M$, by the $\K$-linear maps
\begin{equation*}
E_{m'} \overset{\cdot \chi(m)} \longrightarrow E_{m + m'}.
\end{equation*}
We define a relation on $M$ by setting $m \leq_\sigma m'$ iff $m' - m \in \sigma_M$. One
checks immediately that $\leq_\sigma$ defines a preorder on $M$ with $m \leq_\sigma m'$
and $m' \leq_\sigma m$ iff $m - m' \in \sigma_M^\bot$. We observe that $\chi(m'' - m')
\chi(m' - m) = \chi(m'' - m)$ whenever $m \leq_\sigma m' \leq_\sigma m''$ and $\chi(m - m) = 1$
for every $m \in M$. This way, we can consider every $M$-graded $\ksm$-module in a natural
way as a representation of the preordered set $(M, \leq_\sigma)$.

\begin{definition}
A {\em $\sigma$-family} is a representation of the preordered set $(M, \leq_\sigma)$.
\end{definition}

This definition of $\sigma$-family is equivalent
to the definition given in \cite{perling1}, Definition 5.2. On the other hand, for
every such representation which maps $m$ to some $\K$-vector space $E_m$, we can
construct a $\ksm$-module $E = \bigoplus_{m \in M} E_m$. We have:

\begin{proposition}[\cite{perling1}, Proposition 5.5]\label{gradedrepequiv}
There is an equivalence of categories between the category $M$-graded $\ksm$-modules
and the category of $\sigma$-families.
\end{proposition}

From now on we will not distinguish between an $M$-graded module and its induced representation of $(M, \leq_\sigma)$. By Proposition \ref{popreequivrep}, the category
of $\sigma$-families is equivalent to the category of representations of
$M / \sigma_M^\bot$ with the induced partial order.
By Proposition \ref{popreequivrep} we get an equivalence of categories between the 
category of $\sigma$-families and the category of $\sigma'$-families. By Proposition
\ref{gradedrepequiv} we get equivalently:

\begin{proposition}\label{localequiv}
There is an equivalence of categories between the category of $M$-graded $\ksm$-modules
and the category of $M / \sigma_M^\bot$-graded $\K[\sigma'_M]$-modules.
\end{proposition}

Every $\K$-linear representation of $(M, \leq_\sigma)$ represents a
directed system of $\K$-vector spaces. In \cite{perling1} \S 5.4, colimits
of $\sigma$-families and their relations to the $\tau$-families for $\tau \preceq
\sigma$ have been described. We add some observations which are direct consequences
of the discussion in \cite{perling1} \S 5.4.

\begin{definition}
Let $E$ be an $M$-graded $\ksm$-module. Then we denote $\E$ the colimit
of its associated $\sigma$-family.
\end{definition}

\begin{proposition}\label{directlimitproperties}
\begin{enumerate}[(i)]
\item\label{directlimitpropertiesi} Taking colimits is an exact functor from the
category of $M$-graded $\ksm$-modules to the category of $\K$-vector spaces.
\item\label{directlimitpropertiesii} Let $E$ be an $M$-graded module and $\E$ its
colimit. Then $\dim_k \E = \rk E$.
\end{enumerate}
\end{proposition}

\begin{proof}
(\ref{directlimitpropertiesi}) Just observe that the poset $(M, \leq_\sigma)$ is
filtered and thus colimits are exact.

(\ref{directlimitpropertiesii}) The rank of $E$ coincides with the rank of $E \otimes_\ksm
\K[M]$ and the statement follows from the observation that $\E \cong (E \otimes_\ksm \K[M])_m$
for any $m \in M$.
\end{proof}

If $E$ is torsion-free, then every homomorphism $E_m \overset{\cdot \chi(m)}{\longrightarrow}
E_{m + m'}$ is injective and thus, by general properties of colimits, the canonical
homomorphisms $E_m \longrightarrow \E$ are injective, too. This makes it possible to consider
any torsion-free, $M$-graded $\ksm$-module as a family of subvector spaces of the limit
vector space $\E$ together with some combinatorial information coming from the poset
$(M, \leq_\sigma)$. As has been observed by Klyachko \cite{Kly91} for the case of finitely
generated torsion-free modules, this data can be organized
in terms of multifiltrations (see also \cite{perling1}, \S 5.5) of $\E$. In this work we will
mostly be interested in the more special case, where $E$ is finitely generated and
reflexive. The corresponding structural interpretation by Klyachko in terms of filtered
vector spaces has been stated in Theorem \ref{klythm}.
The morphisms in the category of filtered vector spaces are precisely those vector space
homomorphisms which are compatible with the filtrations in the obvious sense. Given any
family of filtrations $E^k(i)$ as in the Theorem \ref{klythm}, we can reconstruct the
module $E$ by setting
\begin{equation*}
E_m = \bigcap_{k = 1}^n E^k\big(l_k(m)\big).
\end{equation*}

For later use we state the following facts which are straightforward to check:

\begin{proposition}\label{filtaux}
Let $E$ be a finitely generated $M$-graded reflexive $\ksm$-module.
\begin{enumerate}[(i)]
\item\label{filtauxi} $E$ splits into a direct sum  of $M$-graded reflexive $\ksm$-modules
$F \oplus G$ with filtrations $F^k(i)$ and $G^k(i)$ iff $\mathbf{E} \cong
\mathbf{F} \oplus \mathbf{G}$ such that $E^k(i) = F^k(i) \oplus G^k(i)$
for every $k$, $i$.
\item\label{filtauxii} Choose filtrations $F^k(i)$ of $\mathbf{E}$ such that $E^k(i)
\subseteq F^k(i)$ for all $k$, $i$ and denote $F$ the associated reflexive $\ksm$-module.
Then the identity on $\mathbf{E}$ induces an inclusion $E \hookrightarrow F$.
\end{enumerate}
\end{proposition}

\subsection{Matlis duality, injective and projective modules}\label{matlissection}

We denote $\mksmMod$ the category of $M$-graded $\ksm$-modules and $\mksmModf$
its abelian subcategory consisting of modules whose graded components are
finite-dimensional.
Following \cite{GotoWatanabe} \S II.1 (see also \cite{brunsherzog} \S 3.6 and
\cite{SturmfelsMiller} \S 11.3), there exists a natural endofunctor of $\mksmMod$ which
is called the {\em graded Matlis duality}. Explicitly, an object $E$ in $\mksmMod$
is mapped to
\begin{equation*}
\dual{E} = \Hom_\K(E, \K) = \bigoplus_{m \in M} \Hom_\K(E_m, \K) = \bigoplus_{m \in M}
\Hom^M_\ksm(E(m), \K),
\end{equation*}
such that the $M$-grading on $E$ is given by
\begin{equation*}
(\dual{E})_m = \Hom_\K(E_{-m}, \K).
\end{equation*}
The module $\dual{E}$ is called the {\em graded Matlis dual} of $E$.
Its module structure over $\ksm$ for each monomial $\chi(m) \in \ksm$ in
every degree $m' \in M$ is given by
\begin{equation*}
\Hom_\K(\chi(m), \K) : \Hom_\K(E_{-m'}, \K) \longrightarrow \Hom_\K(E_{-m' - m}, \K).
\end{equation*}
The graded Matlis duality functor is exact and satisfies similar properties as its
non-graded counterpart, in particular it induces an anti-equivalence of categories between
finitely generated $M$-graded $\ksm$-modules and artinian $M$-graded $\ksm$-modules
(see \cite{brunsherzog}, \S 3.6). More generally, if $E_m$ is finite-dimensional for
every $m \in M$, then $\dual{\dual{E}} \cong E$ and Matlis duality induces an
autoequivalence of $\mksmModf$. In particular, we get:
\begin{equation*}
\Hom^M_\ksm(E, F) = \Hom^M_\ksm(\dual{F}, \dual{E})
\end{equation*}
for any $E, F$ in $\mksmModf$.

Now let $\tau \preceq \sigma$ be any face. Then $\tau_M = \sigma_M + \Z m_\tau$,
where $m_\tau$ is an element in the relative interior of $\tau'_M$. We see by the
isomorphism $\K[\tau_M] \cong \ksm_{\chi(m_\tau)}$ that $\K[\tau_M]$ is a flat
$\ksm$-module which is contained in $\mksmModf$. In particular, we obtain a family
${\K[\tau_M]}_{\tau \preceq \sigma}$ of flat $\ksm$-modules in $\mksmModf$.
On the other hand, in it was shown in \cite{GotoWatanabe} Theorem 1.3.3
that the indecomposable injective modules in $\mksmMod$ are of the form
\begin{equation*}
I(\K[\tau_M'])(m) \quad \text{ for } \quad \tau \preceq \sigma \text{ and } m \in M,
\end{equation*}
where $I(E)$ denotes the graded injective hull for any module $E$ in $\mksmMod$.
By Matlis duality, these injective modules can explicitly be described as
\begin{equation*}
I(\K[\tau_M'])(m) = \K[\tau_M]\dual{\ }(m)
\end{equation*}
(see \cite{SturmfelsMiller}, \S 11.4). By observing that
$\Hom^M_\ksm\big(-, I(\K[\tau_M'])(m)\big)$ and $\Hom^M_\ksm\big(\K[\tau_M], -\big)$
restrict to endofunctors of $\mksmModf$ and by
the fact that Matlis duality
is an autoequivalence of $\mksmModf$, it is straightforward to show the following:

\begin{proposition}\label{projectives}
The modules $\K[\tau_M](m)$ for $\tau \preceq \sigma$ and $m \in M$ form a complete
set of irreducible projective objects in $\mksmModf$.
\end{proposition}

In particular, the module $\K[M]$ is the unique indecomposable module which is
injective as well as projective in $\mksmModf$.

\begin{remark}
Note that Proposition \ref{projectives} in general is not true for $\mksmMod$.
\end{remark}

\begin{remark}
Note that the category $\mksmModf$ is not a ``good'' category for the construction
of injective or projective resolutions. Example \ref{modfbadexample} below shows
that such resolutions must not necessarily exist in $\mksmModf$. Later on we will restrict
$\mksmModf$ further in order to obtain a ``good'' category which contains finitely
generated modules and their Matlis duals as well as the injective and projective
modules discussed above.
\end{remark}

\begin{example}\label{modfbadexample}
Consider the polynomial ring in one variable $\K[x]$ and the $\Z$-graded module
$E = \bigoplus_{i \in \Z} \K(i)$, where $\K(i)$ denotes the simple module $\K$ shifted
to degree $-i$. Then $E$ does not admit a nontrivial homomorphism from $\K[x, x^{-1}]$
and the first term of a minimal projective resolution of $E$ would necessarily be of
the form
\begin{equation*}
\bigoplus_{i \in \Z} \K[x](i) \longrightarrow E \longrightarrow 0.
\end{equation*}
However, $\bigoplus_{i \in \Z} \K[x](i)$ is not contained in
$\Z$-$\K[x]$-$\operatorname{Mod}_f$.
\end{example}

\subsection{Minimizing projective and injective resolutions}\label{minimizing}

Recall that for a commutative local ring $R$ and any
exact sequence $\cdots \rightarrow F_i \overset{\phi_i} \rightarrow F_{i - 1}$
of free $R$-modules, the presence of a unit element in the matrix representing
$\phi_i$ allows us to do row and column transforms in order to split of one free
summand from $F_i$ and $F_{i - 1}$, respectively. In our setting we consider
a (generalized) projective resolution $P_E$ of $E$ in $\mksmModf$, i.e. a
complex of projective
modules which is everywhere exact except at degree zero and $H_0(P_E) \cong E$.
We assume that $P_{E, i}$ is a finite direct sum of projective modules
for every $i \in \Z$. Then we can write the
differentials $\phi_i : P_{E, i} \rightarrow P_{E, i - 1}$ explicitly as
\begin{equation*}
\bigoplus_{k} \K[\tau_{i, k, M}](m_{i, k}) \overset{\phi_i}{\longrightarrow}
\bigoplus_{j} \K[\tau_{i - 1, j, M}](m_{i - 1, j}),
\end{equation*}
where $\phi_i$ can be represented by a monomial matrix
$\phi_i = \big(\alpha_{ijk} \chi(m_{ijk})\big)_{jk}$ with $\alpha_{ijk} \in \K$.
We have $\alpha_{ijk} = 0$ whenever $\tau_{i - 1, j} \not \preceq \tau_{i, k}$
or $m_{i - 1, j} \not \leq_{\tau_{i, k}} m_{i, k}$. If $\alpha_{ijk} \neq 0$
then $\chi(m_{ijk})$ is uniquely determined up to multiplication by some
$\chi(m)$ with $m \in \tau_{i, k, M}^\bot$.

\begin{lemma}\label{minimizeprojres}
With notation as above assume there are $i, j, k$ with $\alpha{ijk} \neq 0$
and $\tau_{i - 1, j} = \tau_{i, k}$, $ m_{i - 1, j} - m_{i, k} \in \tau_i^\bot$.
Then we can split off $\K[\tau_{i, j, M}](m_{i, j})$ and
$\K[\tau_{i - 1, j, M}](m_{i - 1, j})$ from $P_{E, i}$ and $P_{E, i - 1}$,
respectively, i.e. there is a complex $P_E'$, with $P'_{E, i} \oplus
\K[\tau_{i, j, M}](m_{i, j}) = P_{E, i}$, and $P'_{E, i - 1} \oplus
\K[\tau_{i - 1, j, M}](m_{i - 1, j}) = P_{E, i - 1}$, and $P'_{E, l} = P_{E, l}$
otherwise, such that $P'_E$ is a projective resolution of $E$.
\end{lemma}

\begin{proof}
If our conditions are fulfilled, the splittings $P'_{E, i} \oplus
\K[\tau_{i, k, M}](m_{i, k}) = P_{E, i}$ and $P'_{E, i - 1} \oplus
\K[\tau_{i - 1, j, M}](m_{i - 1, j}) = P_{E, i - 1}$ split off for
every degree $m_{i, k} \leq_{\tau_{i, k}} m$ a one-dimensional vector
space from $P_{E, i}$ and $P_{E, i - 1}$, respectively. So, the complex
stays exact when we compose $\phi_{i +1}$ with the projection from $P_{E, i}$
to $P'_{E, i}$, restrict $\phi_{i - 1}$ to $P'_{E, i - 1}$, and replace
$\phi_i$ by its restriction to $P_{E, i}$ composed with the projection
from $P_{E, i - 1}$ to $P'_{E, i - 1}$.
\end{proof}

Note that in the proof of Lemma \ref{minimizeprojres}, rather than to
refer to explicit row and column transformations, we have made use
of the correspondence between $M$-graded $\ksm$-modules and representations
of $(M, \leq_\sigma)$. Now we observe that Matlis duality maps $P_E$
to $\dual{P}_E$, which is a (generalized) injective resolution of
$\dual{E}$, i.e. Matlis duality
induces a correspondence in $\mksmModf$ between projective resolutions of
$E$ and injective resolutions of $\dual{E}$ and vice versa. So, given
any injective resolution $I_E$ of $E$ with differentials
\begin{equation*}
\bigoplus_{k} \K[\tau_{i, k, M}]\dual{\ }(m_{i, k}) \overset{\phi_i}{\longrightarrow}
\bigoplus_{j} \K[\tau_{i + 1, j, M}]\dual{\ }(m_{i + 1, j}),
\end{equation*}
we can represent $\phi_i$ by the transpose of the corresponding
matrix $\dual{\phi}_{-i}$ in $\dual{I}_E$. Using
this, we get the dual statement for injective resolutions.

\begin{lemma}\label{minimizeinjres}
With notation as above and writing $\phi_i = (\alpha_{ijk} \chi(m_{ijk})$
assume there are $i, j, k$ with $\alpha_{ijk} \neq 0$
and $\tau_{i + 1, j} = \tau_{i, k}$, $ m_{i + 1, j} - m_{i, k} \in \tau_i^\bot$.
Then we can split off $\K[\tau_{i, k, M}]\dual{\ }(m_{i, k})$ and
$\K[\tau_{i + 1, j, M}]\dual{\ }(m_{i + 1, j})$ from $I_{E, i}$ and $I_{E, i + 1}$,
respectively, i.e. there is a complex $I_E'$, with $I'_{E, i} \oplus
\K[\tau_{i, k, M}]\dual{\ }(m_{i, k}) = I_{E, i}$, and $I'_{E, i + 1} \oplus
\K[\tau_{i + 1, j, M}]\dual{\ }(m_{i + 1, j}) = I_{E, i + 1}$, and
$I'_{E, l} = I_{E, l}$ otherwise, such that $I'_E$ is an
injective resolution of $E$.
\end{lemma}

\begin{definition}
We call a projective (respectively injective) resolution {\em minimal} if we cannot
split off summands as in Lemma \ref{minimizeprojres} (respectively Lemma
\ref{minimizeinjres}).
\end{definition}

Having established minimality of projective and injective resolutions, we now
recall the notions of graded Betti and Bass numbers.

\begin{definition}
Let $E$ be a module in $\mksmModf$ and $P_E = \cdots \rightarrow
P_{E, 1} \rightarrow P_{E, 0}$, $I_E = I_E^0 \rightarrow I_E^1 \rightarrow
\cdots$ minimal projective and injective resolutions of $E$, respectively. Then
we have decompositions
\begin{equation*}
P_{E, i} = \bigoplus_{\tau \preceq \sigma} \bigoplus_{m \in M} \K[\tau_M](-m)^{\beta_i(\tau, m)}
\quad \text{ and } \quad
I_E^i = \bigoplus_{\tau \preceq \sigma} \bigoplus_{m \in M} \K[\tau_M]\dual{\ }(-m)^{b^i(\tau, m)}.
\end{equation*}
We denote $\beta_i(\tau, \uc)$ the $i$-th {\em graded Betti number} of degree
$\uc$ with respect to $\tau$ and
$b^i(\tau, \uc)$ the $i$-th {\em graded Bass number} of degree $\uc$
with respect to $\tau$. For $\tau = \sigma$ we also write $\beta_i(m)$ and
$b^i(m)$ instead of $\beta_i(\tau, m)$ and $b^i(\tau, m)$, respectively.
\end{definition}

\subsection{Homogeneous coordinates and divisorial modules}
\label{homogeneouscoordinates}\label{divisorialmodules}

By a well-known construction due to Cox, every toric variety has a so-called
homogeneous coordinate ring. In our situation we consider the polynomial ring
$S = \K[x_1, \dots, x_n]$ considered as monoid ring over $\N^n$ together
with its natural $\Z^n$-grading. By the following exact sequence
\begin{equation}\label{standardsequence}
M \overset{L}{\longrightarrow} \Z^n \overset{\Phi}{\longrightarrow} A \longrightarrow 0,
\end{equation}
where $L$ is defined by $L(m) = \big(l_1(m), \dots, l_n(m)\big)$, we can endow
$S$ with an $A$-grading by setting $\deg_A x^\uc = \Phi(\uc)$ for every monomial
$x^\uc$ in $S$. The image of $M$ in $\Z^n$ coincides with $M / \sigma_M^\bot$.
In light of Proposition \ref{localequiv} we will assume without loss of generality
that $L$ is injective. Then $\ksm$ (rather that $\K[\sigma_M']$) is realized as
the degree zero subring of $S$ with
respect to this grading. Geometrically, this can be interpreted as representation
of $U_\sigma$ as a good quotient of the affine space $\mathbf{A}^n_\K$
by the diagonalizable group scheme $\spec{\K[A]}$. The irreducible torus
invariant
Weil divisors $D_1, \dots, D_n$ on $U_\sigma$ are in one-to-one correspondence
with $l_1, \dots, l_n$ and we can identify $\Z^n$ with the free group generated
over the $D_i$. Also we can identify $A$ with the Weil divisor class group
$A_{d - 1}(U_\sigma)$ and above sequence states that every Weil divisor class
has a torus invariant representative which is determined up to a principal divisor
associated to an element in
$M$ which we interpret as a semi-invariant rational function on $U_\sigma$.

Any $\Z^n$-graded $S$-module $E$ can also be endowed with a natural $A$-grading
by setting $E_\alpha := \bigoplus_{\uc \in \Phi^{-1}(\alpha)} E_{\uc}$ for every $\alpha
\in A$. To better distinguish these gradings we introduce
the following notation:

\begin{definition}
Let $E$ be a $\Z^n$-graded $S$-module. Then for any $\uc \in \Z^n$ we denote
its $A$-degree $\Phi(\uc)$ by
\begin{equation*}
E_{(\uc)} := E_{\Phi(\uc)} = \bigoplus_{m \in M} E_{\uc + L(m)}.
\end{equation*}
\end{definition}

So, every $\alpha \in A$ the $\K$-vector space $E_\alpha$ has a natural structure of
a $\ksm$-module. In particular, taking degree zero is an exact functor
\begin{equation*}
-_{(0)} : \zsMod \longrightarrow \mksmMod.
\end{equation*}

A particular class of $\ksm$-modules arising this way are those coming from projective
and injective modules in $\zsModf$. For any $I \subset \on$ we denote
$x_I := \prod_{i \in I} x_i$ and $S_I := S_{x_I}$ the localization of $S$ at $x_I$. Then
the projective and injective $S$-modules are given by $S_I(\uc)$ and
$\dual{S}_I(\uc)$ respectively (where degree-shifts are given by $\uc \in \Z^n$).
Explicitly, we have
\begin{equation*}
S_I(\uc)_{(0)} = \bigoplus_{m \in M_I^\uc} \K \chi(m) \quad \text{ and } \quad
\dual{S}_I(\uc)_{(0)} = \bigoplus_{-m \in M_I^{-\uc}} \K \chi(m),
\end{equation*}
where
\begin{equation*}
M^{\underline{d}}_I = \{m \in M \mid l_i(m) \geq -d_i \text{ for all }i \in I\}
\end{equation*}
for any $\underline{d} = (d_1, \dots, d_n)�\in \Z^n$.

\begin{definition}
Let $E$ be any module in $\mksmMod$. If $E$ is isomorphic to some $S_I(\uc)_{(0)}$,
then we call $E$ {\em divisorial}. If $E$ is isomorphic to some, $\dual{S}_I(\uc)_{(0)}$
then we call $E$ {\em codivisorial}.
\end{definition}

From our discussion of subsection \ref{matlissection} it follows that a divisorial
module $S_I(\uc)_{(0)}$ is projective in $\mksmModf$ iff $\{l_i \mid i \in I\}$ generate a face of $\sigma$.
Analogously, a codivisorial module $\dual{S}_I(\uc)_{(0)}$ is injective in $\mksmModf$ iff
$\{l_i \mid i \in I\}$ generate a face of $\sigma$. Hence, the classes of divisorial and codivisorial
modules coincide with the classes of projective and injective modules, respectively, in
$\mksmModf$ iff $U_\sigma$ is smooth. If $U_\sigma$ is not smooth then the (co-)divisorial
modules form a strictly larger class.

In the following lemma we collect some general properties of (co-)divisorial modules whose
straightforward check we leave to the reader:

\begin{lemma}\label{injaux}\label{divisorialproperties}
Let $\uc, \uc' \in \Z^n$ and $I, J \subseteq \on$. Then:
\begin{enumerate}[(i)]
\item\label{injauxi}
\begin{equation*}
S_J(\uc') \otimes_S \dual{S}_I(\uc)
\cong
\begin{cases}
\dual{S}_I(\uc + \uc') & \text{ if } J \subseteq I\\
0 & \text{ else}.
\end{cases}
\end{equation*}
\item\label{injauxii} 
\begin{equation*}
\Hom_S\big(S_I(\uc'), S_J(\uc)\big) \cong
\Hom_S\big(\dual{S}_J(\uc'), \dual{S}_I(\uc)\big) \cong
\begin{cases}
S_I(\uc - \uc') & \text{ if } I \subseteq J\\
0 & \text{ else}.
\end{cases}
\end{equation*}
\item $S_I(\uc)_{(0)}$ and $\dual{S}_I(\uc)_{(0)}$ are indecomposable.
\item $S_I(\uc)_{(0)}$ is finitely generated as a $\ksm$-module iff $I = \on$.
\item\label{divisorialpropertiesii} Let $\tau \preceq \sigma$ and $m \in \sigma_M$
be in the relative interior of $\tau'_M$, then the localization $(S_I(\uc)_{(0)})_{\chi(m)}$
equals $S_{I \cup \tau(1)}(\uc)_{(0)}$.
\item
\begin{equation*}
\Hom_\ksm^M(S_I(\uc)_{(0)}, S_J(\uc')_{(0)}) =
\begin{cases}
\K & \text { if } \uc' \leq \uc \text{ and } I \subseteq J\\
0 & \text{ else.}
\end{cases}
\end{equation*}
\item In particular, $S_I(\uc)_{(0)}$ is reflexive iff $I = \on$.
\item\label{divisorialpropertiesv} Let $E$ be a nontrivial, torsion-free, and
finitely generated $M$-graded $\ksm$-module.
Then $\Hom_\ksm(S_I(\uc)_{(0)}, F) = 0$ iff $I \neq \on$.
\end{enumerate}
\end{lemma}

Note that in particular property (\ref{divisorialpropertiesii}) implies that the
sheaves over $U_\sigma$ associated to divisorial modules are quasi-coherent of rank one.

As remarked above, codivisorial modules are not injective in general. However,
the following lemma shows that codivisorial modules are still be useful for computing
local cohomology groups.

\begin{proposition}\label{acycprop}
Let $B \subseteq \ksm$ be a homogeneous ideal, $V \subset U_\sigma$ the corresponding
$T$-invariant closed subscheme of $U_\sigma$, $\uc \in \Z^n$ and $I \subseteq \on$.
Denote $\tau_I$ the minimal face of $\sigma$ such that $I \subseteq \tau_I(1)$.
\begin{enumerate}[(i)]
\item\label{acycpropi}
$\Gamma_V \dual{S}_I(\uc)_{(0)} =
\begin{cases}
\dual{S}_I(\uc)_{(0)} & \text{ if the orbit corresponding to $\tau_I$ is contained in $V$}.\\
0 & \text{ else}.
\end{cases}$
\item\label{acycpropii} The module $\dual{S}_I(\uc)_{(0)}$ is $\Gamma_V$-acyclic.
\item\label{acycpropiii} Consider the preimage $\hat{V} \subset \A^n_\K$ of $V$
under the projection $\A^n_\K \twoheadrightarrow U_\sigma$. Then
$\Gamma_{\hat{V}} \dual{S}_I(\uc) = 0$ iff $\Gamma_V \dual{S}_I(\uc)_{(0)} = 0$.
\end{enumerate}
\end{proposition}

\begin{proof}
For the case that $V$ is the minimal orbit in $U_\sigma$ statements (\ref{acycpropi})
and (\ref{acycpropii}) were shown in \cite{TrungHoa}, Lemma 3.1. We leave the adaption
of this proof to our slightly more general case to the reader.

Statement (\ref{acycpropiii}) follows from (\ref{acycpropi}) and the observation that
the orbit decomposition of $\hat{V}$ with respect to the toric structure of $\A^n_\K$
corresponds one-to-one to subsets $I$ of $\on$ such that the orbit in $U_\sigma$
corresponding to $\tau_I$ is contained in $V$.
\end{proof}

Later on we will use codivisorial resolutions in order to compute local cohomology
with respect to to $V$ of general $M$-graded $\ksm$-Modules. As a first step, we will
first determine the
local cohomology of divisorial modules. Following \cite{GotoWatanabe} \S 3.2, we
can consider the Grothendieck-Cousin complex of any divisorial module $S_I(\uc)_{(0)}$.
For this, we observe that we can decompose
\begin{equation*}
S_I(\uc) \cong \left(\Big( \bigotimes_{i \in I} \K[x_i, x_i^{-1}] \Big)
\otimes \Big( \bigotimes_{i \in \on \setminus I} \K[x_i](c_i) \Big)\right),
\end{equation*}
where all tensor products are over $\K$, all $\K[x_i, x_i^{-1}]$, $\K[x_i]$
are considered as $\Z$-graded, and $\uc = (c_1, \dots, c_n)$.
Replacing the factors on the right hand side by the quasiisomorphic complexes
$\K[x_i, x_i^{-1}] \rightarrow \K[x_i, x_i^{-1}] / \big(\K[x_i](c_i)\big)$
for every $i \notin I$, we get an isomorphism of complexes
\begin{equation*}
S_I(\uc) \cong \left(\Big( \bigotimes_{i \in I} \K[x_i, x_i^{-1}] \Big)
\otimes \Big( \bigotimes_{i \in \on \setminus I} \K[x_i, x_i^{-1}] \rightarrow
\K[x_i, x_i^{-1}] / \big(\K[x_i](c_i)\big) \Big)\right) =: C^\bullet_{I, \uc}
\end{equation*}
where now the right hand side denotes the total tensor product chain complex.
Its degree-zero part $(C^\bullet_{I, \uc})_{(0)}$ is a codivisorial  --- and by Proposition
\ref{acycprop} therefore a $\Gamma_V$-acyclic --- resolution of $S_I(\uc)_{(0)}$.
Denote $\hat{\sigma}$ the combinatorial simplex on the set $\on$ and for
any $m \in M$ denote $\hat{\sigma}_m$ its full
subsimplex supported on $i \in I$ with $l_i(m) < -c_i$.
By inspection it turns out that the graded piece $\big((C^\bullet_{I, \uc})_{(0)}\big)_m$
coincides with simplicial cochain complex of $\hat{\sigma}_{\underline{b}}$ over $\K$. 
Now, for any torus invariant
closed subset $V \subseteq U_\sigma$, its complement is again a toric variety
described by a subfan $\sigma_V$ of $\sigma$. We denote $\hat{\sigma}_{V, m}$
the simplicial subcomplex of $\hat{\sigma}_m$ such that
$I \in \hat{\sigma}_{V, m}$ implies $\tau_I \preceq \sigma_V$
with the notation of Proposition \ref{acycprop}.
Now we consider the graded
decomposition $H^i_V(S_I(\uc)_{(0)}) \cong \bigoplus_{m \in M} H^i_V(S_I(\uc)_{(0)})_m$.
Applying $\Gamma_V$ to $(C^\bullet_{I, \uc})_{(0)}$ together with
standard arguments involving the long exact cohomology
sequence associated to $\hat{\sigma}_m$ and
$\hat{\sigma}_{V, m}$ imply the following variation of a
well-known standard result:

\begin{proposition}
For $I, \uc$ as above and $i \geq 0$ we have for every $m \in M$
\begin{equation*}
H^i_V(S_I(\uc)_{(0)})_m \cong H^{i - 2}(\hat{\sigma}_{V, m}; \K),
\end{equation*}
where $H^{i - 2}(\hat{\sigma}_{V, m}; \K)$ denotes the $i - 2$-th reduced cohomology
of the simplicial complex $\hat{\sigma}_{V, m}$ with coefficients in $\K$.
\end{proposition}

\section{Resolutions}\label{resolutionsection}

In this section we will develop a general framework for divisorial and codivisorial
resolutions of $M$-graded $\ksm$-modules. For this we will first consider $\Z^n$-graded
modules over the polynomial ring $S = \K[x_1, \dots, x_n]$ and introduce $\lcm$- and
$\gcd$-lattices for such modules in subsection \ref{lcmlattices}. In subsection
\ref{reflexivelcmlattices} we identify the minimal admissible $\lcm$-lattice of a reflexive
module with the intersection poset of the vector space arrangement generated by its
associated filtrations. In subsection \ref{combinatoriallyfinite} we introduce the category
of combinatorially finite modules. Subsections \ref{resolutions} and \ref{reflres} form
the central parts of this section. In subsetion \ref{resolutions} we describe minimal projective
resolutions of combinatorially finite modules in terms of their associated $\lcm$-lattices.
In subsection \ref{reflres} we characterize these resolutions for reflexive modules as
resolutions of vector space arrangements. In subsection \ref{duality} we give an interpretation
of Miller's results on duality of projective and injective resolutions in terms of an autoequivalence
of the derived category of combinatorially finite modules. Based on this, we discuss the
computation of local cohomologies in subsection \ref{loccohomcomp}.

\subsection{lcm- and gcd-lattices}\label{lcmlattices}

Let $\uc = (c_1, \dots, c_n), \uc' = (c_1', \dots, c_n') \in \Z^n$, then the least
common multiple of $\uc$ and $\uc'$ is defined as the exponent of the least common
multiple of the monomials $x^\uc$ and $x^{\uc'}$, i.e. $\lcm\{\uc, \uc'\} = (\max\{c_1, c_1'\}, \dots, \max\{c_n, c_n'\})$. Analogously, the greatest common divisor is defined
as $\gcd\{\uc, \uc'\} = (\min\{c_1, c_1'\}, \dots, \min\{c_n, c_n'\})$. We set
$\Zbar := \Z \cup \{-\infty, \infty\}$ which is totally ordered by $-\infty < n < \infty$
for all $n \in \Z$ and $-\infty$ denotes an actual sign change, i.e. $-\infty = -(\infty)$.
We naturally extend the notions $\lcm\{\uc, \uc'\}$ and
$\gcd\{\uc, \uc'\}$ to any $\uc, \uc' \in \Zbar^n$. Also, the partial order on
$\Z^n$ extends naturally to a partial order on $\Zbar^n$.

Any object $E$ in $\zsMod$ can be extended to a representation of $(\Zbar^n, \leq)$
by setting for every $\n \in \Zbar^n$
\begin{equation*}
\bar{E}_\n := \underset{\leftarrow}{\lim}\ E_{\n'},
\end{equation*}
where the limit is taken over all $\n' \in \Z^n$ with $\n \leq \n'$. We set $\bar{E}_\n$
to zero if this set is empty. This construction establishes a functor
\begin{equation*}
\bar{\ } : \zsmod \longrightarrow \text{Representations of } (\Zbar^n, \leq),
\quad E \mapsto \bar{E}.
\end{equation*}
Conversely, if we consider $\Z^n$ as a topological subspace of $\Zbar^n$ and denote
the inclusion $\iota : \Z^n \hookrightarrow \Zbar^n$, we obtain by restriction of sheaves
the functor
\begin{equation*}
\iota^{-1} : \text{Representations of } (\Zbar^n, \leq) \longrightarrow \zsmod,
\quad \bar{E} \mapsto \iota^{-1} \bar{E}.
\end{equation*}
By the universal property of limits, the pair of functors $\bar{\ }$ and $\iota^{-1}$
establishes an equivalence of categories between $\zsmod$ and its essential image in
the category of representations of $\Zbar^n$.

The partially ordered set $\Zbar^n$ is a lattice with meet $\wedge$ and join $\vee$
being $\gcd$ and $\lcm$, respectively. Any subset of $\Zbar^n$ which is closed under
joins (respectively meets) is called {\em join sublattice} (respectively {\em meet
sublattice}). However in our context the following terms are more customary.

\begin{definition}
We call {\em lcm-lattice} any subset $\cL$ of $\Zbar^n$ which is closed under taking
$\lcm$ and moreover for any element $(c_1, \dots, c_n)$ in $\cL$ we have
$c_i < \infty$ for all $i$.
We call {\em gcd-lattice} any subset $\cL$ of $\Zbar^n$ which is closed under taking
$\gcd$ and moreover for any element $(c_1, \dots, c_n)$ in $\cL$ we have
$c_i > -\infty$ for all $i$.
\end{definition}

Let $\iota_\p : \p \hookrightarrow \Zbar^n$ be any subposet considered as topological
space by its subspace topology and $E$ in $\zsmod$. Then the restriction of
$\bar{E}$ to $\p$ is given by the sheaf theoretical pullback $\iota_\p^{-1}\bar{E}$.
In terms of poset representations, this simply coincides with the restriction of
$\bar{E}$ to $\p$. For the case that $\p$ is an $\lcm$-lattice we define a functor
which we think of as compression of the relevant information
contained in $\bar{E}$ to a smaller -- possibly finite -- poset.

\begin{definition}\label{zipdef}
Let $\iota_\cL : \cL \hookrightarrow \Zbar^n$ be an $\lcm$-lattice. Then for any $E$ in $\zsmod$
we define
\begin{equation*}
\zip^\cL E := \lcmzip^\cL E := \iota^{-1}_\cL \bar{E}.
\end{equation*}
\end{definition}
Below we will introduce the Matlis dual notion $\gcdzip^\cG$ but we will mostly
be working with $\lcmzip^\cL$. Therefore we will usually drop the subscript ``$\lcm$''
if there is no ambiguity.

Now, given a representation $F$ of an $\lcm$-lattice $\cL$, it is possible to extend 
this representation
to a representation $F'$ of $\Zbar^n$ as follows. Given any $\uc \in \Zbar^n$,
we have two possibilities: either there exists no element $\uc'$ in $\cL$ with
$\uc' \leq \uc$, or there exists a unique maximal element $\max(\uc) \in \cL$
with $\max(\uc) \leq \uc$, which is the $\lcm$ of all $\uc' \in \cL$ with
$\uc' \leq \uc$. In the first case we set $F'_\uc := 0$. In the second case
we set $F'_\uc := F_{\max(\uc)}$. For any $\uc \leq \uc'$ such that $F'_\uc
\neq 0$, we set as morphism $F'(\uc, \uc') := F\big(\max(\uc), \max(\uc')\big)$
(and zero otherwise). This way we
get a $\K$-linear representation of $\Zbar^n$.

\begin{definition}\label{unzipdef}
Let $\cL \subseteq \Zbar^n$ be an $\lcm$-lattice, $F$ a representation of $\cL$
and $F'$ its extension to $\Zbar^n$ as constructed above. Then we denote
\begin{equation*}
\unzip^\cL F := \lcmunzip^\cL F := F'.
\end{equation*}
\end{definition}

In general, applying the functor $\zip^\cL$ to any module $E$ will destroy a lot
of information about $E$ such that it is not possible to reconstruct $E$ from
its restriction to $\cL$.

\begin{definition}
Let $E$ be in $\zsmod$ and $\cL \subseteq \Zbar^n$ an
$\lcm$-lattice. Then we call $\cL$ an {\em $E$-admissible} (or simply admissible)
$\lcm$-lattice if $E \cong \unzip^\cL \circ \zip^\cL E$.
\end{definition}

\begin{example}\label{admissibleexample1}
Let $I \subset \on$ and $\uc \in \Z^n$. Then the module $S_I(\uc)$ admits
an admissible $\lcm$-lattice which consists of only one element $p_{\uc, I}
= (p_1, \dots, p_n) \in \Zbar^n$,
where
\begin{equation*}
p_i =
\begin{cases}
-\infty & \text{ for } i \in I\\
c_i & \text{else}.
\end{cases}
\end{equation*}
The representation $\zip^{\{p_{\uc, I}\}}S_I(\uc)$ then associates to
$p_{\uc, I}$ the vector space $\K$. Conversely, every element $\underline{p}$
in $\Zbar^n$ with $p_i < \infty$ for all $i$ together with the representation
$\underline{p} \mapsto \K$ gives rise to a module $S_I(\uc)$, where
$I = \{i \in \on \mid p_i = -\infty\}$ and $\uc \in \Z^n$ any element such that
$c_i = p_i$ for $i \notin I$.
\end{example}

\begin{example}\label{admissibleexample2}
Let $S_I(\uc)$ be as in the previous example. Then for the Matlis duals $\check{S}_I(-\uc)$ 
we get a more complicated picture. A minimal admissible $\lcm$-lattice $\cL_{\uc, I}$ is the
$\lcm$-lattice generated by $(-\infty, \dots, -\infty)$ and $\{\underline{p}^i\}_{i \notin I}$
where
\begin{equation*}
p^i_j =
\begin{cases}
-c_i + 1 & \text{ for } j = i\\
-\infty & \text{ else}.
\end{cases}
\end{equation*}
The representation $\zip^{\cL_{\uc, I}} \check{S}_I(-\uc)$ associates $\K$ to
$(-\infty, \dots, -\infty)$ and $0$ to all other elements of $\cL_{\uc, I}$.
\end{example}

\begin{example}\label{admissibleexample3}
Let $E$ in $\zsMod$ be finitely generated and torsion free and denote $e_1, \dots,
e_t$ a minimal set of homogeneous generators of $E$ with degrees $\uc_1, \dots, \uc_t$.
Then the $\lcm$-lattice generated by $\uc_1, \dots, \uc_t$ is admissible for $E$.
In case that $E$ is a monomial ideal, this $\lcm$-lattice coincides with the
classical $\lcm$-lattice of \cite{GPW99}. For general torsion free modules it
coincides with that of \cite{CharalambousTchernev03}.
\end{example}

\begin{example}\label{artinianexample}
Let $E = \K[x, y] / (x^2, xy, y^2)$. A minimal admissible $\lcm$-lattice is generated
by $\{0, (2, 0), (1, 1), (0,2)\}$ in $\Z^2$.
\end{example}

The proof of the following proposition yields a general construction of admissible
posets:

\begin{proposition}\label{admissibleexists}
Any module in $\zsMod$ admits an admissible poset.
\end{proposition}

\begin{proof}
Let $E$ be in $\zsMod$ and $\bar{E}$ its extension to $\Zbar^n$. Then for any $\uc \in
\Z^n$ we denote $I_E(\uc)$ the set of all those elements $\uc' \in \Zbar^n$ which
are minimal with the property that for all $\uc'' \in \Zbar^n$ with
$\uc' \leq \uc'' \leq \uc$ the homomorphisms $\bar{E}_{\uc'} \rightarrow \bar{E}_{\uc''}$
and $\bar{E}_{\uc''} \rightarrow \bar{E}_\uc$ are isomorphisms. Denote $\cL_E$ the
$\lcm$-lattice generated by $\bigcup_{\uc \in \Z^n} I_E(\uc)$. We claim that
$\cL_E$ is $E$-admissible. By construction we get that
$(\unzip^{\cL_E} \circ \zip^{\cL_E} E)_\uc \cong E_\uc \cong E_{\max(\uc)}$
for every $\uc \in \Z^n$ (with notation as used in the context of Definition
\ref{unzipdef}). We establish an isomorphism of representations by explicitly
setting this isomorphism equal to the isomomorphism $E_{\max(\uc)} \rightarrow
E_\uc$. It is straightforward to check that this indeed induces an equivalence
of representations.
\end{proof}

\begin{remark}
Though the procedure in the proof of Proposition \ref{admissibleexists} always
delivers an admissible $\lcm$-lattice, this might not be an optimal (or minimal)
choice in general. For the module of Example \ref{artinianexample} we get
that $\cL_E$ is generated by $\{(-\infty, -\infty), (2, -\infty), (-\infty, 2), 0, (1, 1)\}$
which is larger than the minimal $\lcm$-lattice given in Example \ref{artinianexample}.
\end{remark}

\subsection{$\lcm$-lattices of reflexive modules}\label{reflexivelcmlattices}

The most important class of modules we want to understand are the
finitely generated reflexive modules in $\zsMod$. Let $E$ be such a module,
described by a family of full filtrations $\cdots \subseteq E^k(i)
\subseteq E^k(i + 1) \subseteq \cdots \subseteq \mathbf{E}$ for every $k \in \on$.
Any such filtration is determined by its underlying flag of subvector spaces
of $\mathbf{E}$: for any $k$ denote $i^k_1 < \cdots < i^k_{t_k}$ the maximal
sequence of integers such that $\dim E^k(i^k_j - 1) < \dim E^k(i^k_j)$ for any
$1 \leq j \leq t_k$. Then the vector spaces $E^k(i^k_1) \subsetneq \cdots
\subsetneq E^k(i^k_{t_k})$ form a (partial) flag in $\mathbf{E}$. In particular,
the set of subvector spaces $\{E^k(i^k_j) \mid k \in \on, 1 \leq j \leq t_k\}$
forms a subvector space arrangement in $\mathbf{E}$. We denote by $\mathcal{V}_E$
the subvector space arrangement in $\mathbf{E}$ which is generated by all
intersections of the $E^k(i^k_j)$. $\mathcal{V}_E$ clearly forms a partially
ordered set with respect to inclusion. To avoid cumbersome notation, we
introduce the following convention.

\begin{convention}\label{arrangementconvention}
Let $\mathcal{V}$ (or $\mathcal{V}_E, \mathcal{V}_{F_i}, \dots$), denote a vector
space arrangement, i.e. a family of subvector space of some fixed vector space which
is closed under taking intersections.
Then by abuse of notation, we use the same symbol $\mathcal{V}$ (or $\mathcal{V}_E,
\mathcal{V}_{F_i}, \dots$) to denote its underlying poset.
\end{convention}

With is convention, we identify the {\em partially ordered set} $\mathcal{V}_E$
with a subset of $\Z^n$ by the following map:
\begin{equation*}
\mathcal{V}_E \ni X \mapsto \underline{i}^X := (i_1^X, \dots, i_n^X) \in \Z^n,
\end{equation*}
where $i_k^X := \min\{i_j^k \mid X \subseteq E^k(i_j^k)\}$.
This map is order preserving and we can thus identify the poset underlying
$\mathcal{V}_E$ with a subposet of $\Z^n$. We even get:

\begin{proposition}\label{reflexivelcmlattice}
Let $E$ be a finitely generated reflexive module in $\zsMod$. Then the set
$\mathcal{V}_E$ as defined above is a minimal $E$-admissible $\lcm$-lattice.
\end{proposition}

\begin{proof}
We first show that $\mathcal{V}_E$ is an $\lcm$-lattice. Let $X, Y \in \mathcal{V}_E$,
then $\lcm\{\underline{i}^X, \underline{i}^Y\} = (\max\{i_1^X,$ $i_1^Y\},$ $\dots,
\max\{i_n^X, i_n^Y\}) =: (j_1, \dots, j_n)$. Denote $Z :=
\bigcap_{k = 1}^n E^k(j_k)$. We claim that $j_k = i_k^Z$ for every $k$. Assume
there exists some $k$ such that $i_k^Z < j_k$. Then either $i_k^Z < i_k^X$ or
$i_k^Z < i_k^Y$. But then either $X \nsubseteq E^k(i_k^Z)$ or $Y \nsubseteq E^k(i_k^Z)$.
But then either $X \nsubseteq Z$ or $Y \nsubseteq Z$ which is a contradiction,
as $Z$ by construction contains both $X$ and $Y$. Therefore $\mathcal{V}_E$ is
an $\lcm$-lattice.

To check that $\mathcal{V}_E$ is admissible, it suffices to verify that
$(\unzip^{\mathcal{V}_E} \circ \zip^{\mathcal{V}_E} E)_\uc$ is isomorphic to
$E_\uc$ for every $\uc \in \Z^n$. As $\mathcal{V}_E$ does not contain any
element $\uc$ such that $c_k < i^k_1$ for any $k$ and $E_\uc = 0$ for
any such $\uc$, it suffices to assume that $\uc \geq (i^1_1, \dots, i^n_1$).
Consider any graded component $E_\uc$. We denote $i^k_\uc :=
\min\{i^k_j \mid E_\uc \subseteq E^k(i^k_j)\}$ and $\underline{i}_\uc =
(i^1_\uc, \dots, i^n_\uc)$. We observe that $E_{\underline{i}_\uc}$ is an
element of $\mathcal{V}_E$ and $E_{\underline{i}_\uc} \cong E_\uc$.
Also, $\underline{i}_\uc$ is the $\lcm$ of all $\underline{i}_X$ such
that $X \subseteq E_\uc$.
\end{proof}

For $E$ reflexive, among all possible choices for $E$-admissible $\lcm$-lattices,
the vector space arrangement $\mathcal{V}_E$ is a distinguished choice:

\begin{definition}
Let $E$ and $\mathcal{V}_E$ be as above, then we call $\mathcal{V}_E$ the
{\em canonical} $E$-admissible $\lcm$-lattice.
\end{definition}

\subsection{Combinatorially finite modules}\label{combinatoriallyfinite}

Now we want establish for any given finite $\lcm$-lattice $\cL$ an equivalence
of categories between representations of $\cL$ and $S$-modules for which $\cL$
is admissible.

\begin{proposition}\label{admissibleequivalence}
Let $\cL$ be a finite $\lcm$-lattice in $\Zbar^n$. Then the functors $\zip^\cL$
and $\unzip^\cL$ establish an equivalence of categories between the category
of representations of $\cL$ and the category of modules in $\zsMod$ for which
$\cL$ is admissible.
\end{proposition}

\begin{proof}
By definition, we have $E \cong \unzip^\cL \circ \zip^\cL E$. Now let $F$ be
any representation of $\cL$. We have to show that $F_\uc = (\zip^\cL \circ \unzip^\cL F)_\uc$
for every $\uc \in \cL$. This is immediately clear if $\uc \in \Z^n$. In the
case that $\uc \in \Zbar^n \setminus \Z^n$, we have $(\zip^\cL \circ \unzip^\cL F)_\uc
\cong \ilim\, (\unzip^\cL)_{\uc'}$, where the limit runs over all $\uc' \geq \uc$.
By the finiteness of $\cL$, there is a nonempty region in $\Zbar^n$ of which
$\uc$ is the unique minimal element and $(\bar{\unzip^\cL F})_{\uc'} = F_\uc$
for all $\uc'$ in this region. Therefore $(\bar{\unzip^\cL F})_\uc$ and thus
$(\zip^\cL \circ \unzip^\cL F)_\uc$ must coincide with $F_\uc$.
\end{proof}

In the sequel we will only consider modules in $\zsMod$ which admit a finite
admissible $\lcm$-lattice and whose graded components are finite-dimensional:

\begin{definition}
We call an $S$-module $E$ {\em combinatorially finite}
if it is in $\zsModf$ and admits a finite $E$-admissible $\lcm$-lattice $\cL
\subseteq \Zbar^n$. We denote {\em $\zsmod$} the category of combinatorially
finite $S$-modules.
\end{definition}

We have seen in Examples \ref{admissibleexample1} and \ref{admissibleexample2} that
$\zsmod$ contains the injective and projective objects of $\zsModf$. The following
proposition whose proof we leave to the reader shows that $\zsmod$ moreover contains finitely
generated modules and their Matlis duals:

\begin{proposition}\label{combfinprop}
The category of combinatorially finite $S$-modules is abelian and closed under $\Hom_S$,
$\otimes_S$, localization at monomials, and Matlis duality.
\end{proposition}

We can now introduce dual notion of associated $\gcd$-lattices. Note that a subposet
$\cG \subseteq \bar{\Z}^n$ is a $\gcd$-lattice iff $-\cG$ is an $\lcm$-lattice.

\begin{definition}\label{gcdadmissibledef}
Let $E$ be an $S$-module in $\zsmod$ and $\cG$ a finite $\gcd$-lattice in $\Zbar^n$.
Then $\cG$ is {\em $E$-admissible} if $-\cG$ is an $\check{E}$-admissible $\lcm$-lattice.
\end{definition}

We can also define zip and unzip for $\gcd$-lattices.

\begin{definition}
For any subposet $\p$ of $\Zbar^n$ and any representation $F$ of $\p$, denote
by $F^\op$ the dual representation of the poset $-\p$, which maps $\uc$ to $\Hom_\K(F_\uc, \K)$.
Then for any $E$ in $\zsmod$ and any $\gcd$-lattice $\cG$ we set
\begin{equation*}
\gcdzip^\cG E := \left(\lcmzip^{-\cG}\check{E}\right)^\op.
\end{equation*}
For a representation $F$ of $\cG$ we set
\begin{equation*}
\gcdunzip^\cG F := \left(\lcmunzip^{-\cG}F^\op\right)\check{\ }.
\end{equation*}
\end{definition}

We leave it to the reader to formulate the $\gcd$-version of Proposition
\ref{admissibleequivalence}.

Now we can introduce the notion of combinatorial finiteness for general
rings $\ksm$ with homogeneous coordinate ring $S$:

\begin{definition}
A module $E$ in $\mksmModf$ is called {\em combinatorially finite} if
there exists some $F \in \zsmod$ such that $E \cong F_0$. We denote $\mksmmod$
the full category of $\mksmModf$ of combinatorially finite modules.
\end{definition}

At this point we could, similar as in Proposition \ref{combfinprop}, try to proceed by
investigating general properties of the category $\mksmmod$.
However, our approach instead will be to construct resolutions for
a given $\ksm$-module $E$ by constructing a resolution for some appropriate
$S$-module $F$ with $F_0 = E$.

We conclude this subsection with the following observation:

\begin{theorem}
The category $\mksmmod$ is a Krull-Schmidt category.
\end{theorem}

\begin{proof}
It suffices to show that for any combinatorially finite modules $E, F$ the
$\K$-vector space $\Hom^M_\ksm(E, F)$ is finite-dimensional. For this, we choose
combinatorially finite $S$-modules $E'$ and $F'$ with $E'_{(0)} = E$ and
$F'_{(0)} = F$, respectively, and a finite $\lcm$-lattice $\cL$ which is
admissible for both $E'$ and $F'$. Then $\Hom_\ksm^M(E, F)$ is a subvector
space of $\bigoplus_{\uc \in \cL} \ilim E_m$, where $\ilim E_m$ runs over
all $m \in M$ with $\uc \leq m$ and hence is finite-dimensional.
\end{proof}

\subsection{Using $\lcm$- and $\gcd$-lattices to compute resolutions}\label{resolutions}

In this subsection we will show that we can construct projective or injective resolutions
for combinatorially finite $\ksm$-modules by computing the corresponding resolutions
in an appropriate category of representations of admissible posets. The category of
representations of a poset is equivalent to the category of modules over the
incidence algebra of the associated Hasse diagram. We will present some basic facts
about such representations and refer to \cite{Ass06} for general overview of the
theory of path algebras.

Let $\p$ be a finite poset. There is a bijection between the
elements of $\p$ and the indecomposable projective object in the category
of representations of $\p$, where
for any $x \in \p$ its unique associated projective object $P_x$ is given by:
\begin{equation*}
P_{x, y} =
\begin{cases}
\K & \text{ if } x \leq y\\
0 & \text{ else}
\end{cases}
\end{equation*}
together with identity homomorphisms $P_{x, y} \rightarrow P_{x, z}$, whenever $x \leq y
\leq z$. Similarly, there is a bijection between $\p$ and the indecomposable injective
objects. For $x \in \p$ its injective module $I_x$ is given by:
\begin{equation*}
I_{x, y} =
\begin{cases}
\K & \text{ if } y \leq x\\
0 & \text{ else}
\end{cases}
\end{equation*}
together with identity homomorphisms $I_{x, y} \rightarrow I_{x, z}$, whenever $y \leq z
\leq x$.

Every representation of $\p$ admits a finite projective as well as a finite injective
resolution (see Proposition \ref{globaldimbound} below).
As these resolutions will be crucial for our applications later on, we give now an
explicit algorithmical description on how to obtain minimal projective resolutions.
Let $F$ be any representation of $\p$. For any $x \in \p$, we can split the vector
space $F_x$ as follows:
\begin{equation*}
F_x \cong F_{x, \leq} \oplus F_{x, >}
\end{equation*}
where $F_{x, \leq} = \sum_{y < x} F(y, x)(F_y)$. By fixing a basis of $F_{x, >}$ we get an
isomorphism $\K^{n^0_x} \cong F_{x, >}$ which gives rise to a natural homomorphism
of $\p$-representations
\begin{equation*}
P_x^{n_x} \rightarrow F
\end{equation*}
mapping the generators of $(P_x^{n^0_x})_x$ to the basis of of $F_{x, >}$. Consequently,
we get a short exact sequence of $\p$-representations
\begin{equation*}
0 \longrightarrow F^1 \longrightarrow \bigoplus_{x \in \p} P_x^{n^0_x}
\longrightarrow F \longrightarrow 0.
\end{equation*}
The kernel $F^1$ then is the $1$st syzygy of $F$. By iterating, we obtain a projective
resolution:
\begin{equation*}
0 \longrightarrow \bigoplus_{x \in \p} P_x^{n^t_x} \longrightarrow \cdots
\longrightarrow \bigoplus_{x \in \p} P_x^{n^0_x} \longrightarrow F \longrightarrow 0.
\end{equation*}
This resolution is minimal by construction. Similarly,
we can split $F_x \cong F^{x, <} \oplus F^{x, \geq}$, where $F^{x, <} =
\bigcap_{x < y} \ker F(x, y)$ and $F^{x, \geq} \cong \K^{m_x^0}$ for $m_x^0 = \dim
F^{x, \geq}$. By iterating, we get a minimal injective resolution:
\begin{equation*}
0 \longrightarrow F \longrightarrow \bigoplus_{x \in \p} I_x^{m^0_x} \longrightarrow \cdots
\longrightarrow \bigoplus_{x \in \p} I_x^{m^r_x} \longrightarrow 0.
\end{equation*}
From these constructions we can read off the following general property of $\p$-representations:

\begin{proposition}\label{globaldimbound}
Let $\p$ be a finite poset and $F$ a representation of $\p$. Then both projective and
injective dimension of $\p$ are bounded by the maximal length of a chain in $\p$ minus
one.
\end{proposition}

Note that by the general theory of path algebras it follows that projective and injective
dimension coincide.
Theorem \ref{globaldimbound3} below will give even more effective bounds for posets which
can be realized as $\lcm$-lattices.

\begin{proof}
Let $x \in \p$ be minimal with the property that $F_x \neq 0$. Then, by above
construction we get that $F^1_x = 0$. So, if $x_1 < x_2 \cdots < x_s$ is a maximal
chain in $\p$, then in the iteration we get for the $i$-th syzygy $F^i$ that
$F^i_{x_i} = 0$. This argument applies analogously to minimal injective
resolutions.
\end{proof}

Now let $E$ be any module in $\zsmod$, and $\cL \subset \Zbar^n$ a finite
$E$-admissible $\lcm$-lattice. Then, as above, we can compute a minimal projective
resolution
\begin{equation*}
0 \longrightarrow \bigoplus_{\uc \in \cL} P_\uc^{n^t_\uc} \longrightarrow \cdots
\longrightarrow \bigoplus_{\uc \in \cL} P_\uc^{n^0_\uc} \longrightarrow \lcmzip^\cL E
\longrightarrow 0.
\end{equation*}
Unzipping yields an exact sequence
\begin{equation}\label{projres}
0 \longrightarrow \bigoplus_{\uc \in \cL} \lcmunzip^\cL P_\uc^{n^t_\uc} \longrightarrow
\cdots \longrightarrow \bigoplus_{\uc \in \cL} \lcmunzip^\cL P_\uc^{n^0_\uc} \longrightarrow
E \longrightarrow 0.
\end{equation}
Observing that $\lcmunzip^\cL P_\uc \cong S_I(-\uc')$ with $I = \{i \in \on \mid c_i =
-\infty\}$ and $\uc' = (c_1', \dots, c_n') \in \Z^n$ such that $c_i' = c_i$ whenever
$i \notin \on$ (see Example \ref{admissibleexample1}), we see that we have produced
a minimal free resolution of $E$.
Similarly, if $\mathcal{K}$ is a finite $E$-admissible $\gcd$-lattice, we obtain
an injective resolution of $E$ and a codivisorial resolution of $E_0$:
\begin{equation}\label{injres}
0 \longrightarrow E \longrightarrow \bigoplus_{\uc \in \cL}
\gcdunzip^\mathcal{K} I_\uc^{m^0_\uc} \longrightarrow \cdots
\longrightarrow \bigoplus_{\uc \in \cL} \gcdunzip^\mathcal{K}
I_\uc^{m^r_\uc} \longrightarrow 0.
\end{equation}

If we consider $S$ as homogeneous coordinate ring
for $\ksm$, then the degree zero parts  $(\lcmunzip^\cL P_\uc^{n^i_\uc})_{(0)}$
and $(\gcdunzip^\mathcal{K} I_\uc^{m^0_\uc})_{(0)}$ are divisorial, respectively
codivisorial $\ksm$-modules. By the exactness of taking degree zero we thus get:

\begin{theorem}
Let $F$ be a combinatorially finite $\ksm$-module and $E$ a combinatorially finite
$S$-module with $E_{(0)} = F$. Then the degree zero parts of complexes (\ref{projres})
and (\ref{injres}) yield divisorial, respectively codivisorial resolutions of $F$ over
$\ksm$.
\end{theorem}

A first application of these constructions is a combinatorial characterization
of projective and injective dimension of a combinatorially finite $S$-module.

\begin{theorem}\label{globaldimbound2}
Let $E$ be a module in $\zsmod$. Then both the projective and injective dimension
of $E$ are bounded by the maximal chain in an $E$-admissible finite $\lcm$-
and $\gcd$-lattice, respectively.
\end{theorem}

\begin{proof}
Follows immediately from Proposition \ref{globaldimbound}.
\end{proof}

The following observation may be of independent interest for the representation theory
of incidence algebras:

\begin{theorem}\label{globaldimbound3}
Let $\p$ be a finite poset. If $\p$ is isomorphic either to an
$\lcm$-lattice or a $\gcd$-lattice in some $\Z^n$, then both the projective and the
injective dimension of $\p$ are bounded by $n$.
\end{theorem}

\begin{proof}
Projective and injective dimensions of any representation $F$ of $\p$ coincides
with the projective and injective dimensions of $\lcmunzip^\p F$ and $\gcdunzip^\p F$,
respectively, as $S$-modules.
The projective dimension of $S$ equals its injective dimension. By the Hilbert
syzygy theorem, the projective dimension of $S$ is $n$.
\end{proof}

\begin{remark}
Note that for Theorem \ref{globaldimbound3} it is not necessary to consider
embeddings into $\Zbar^n$. Because $\p$ is finite, then, if we can an embedding
of $\p$ into $\Zbar^n$, we can also find an embedding into $\Z^n$.
\end{remark}

\subsection{Reflexive modules and resolutions of vector space arrangements}\label{reflres}

Let $E$ be a reflexive module in $\zsmod$ with filtrations $E^k(i)$ and
numbers $i_1^k < \cdots < i^k_{t_k}$ as used in Proposition
\ref{reflexivelcmlattice}, denoting the steps in the filtrations.
The set of vector spaces $E_\uc = \bigcap_{k \in \on} E^k(c_i) \subseteq \mathbf{E}$
form a vector space arrangement $\mathcal{V}_E$ in $\mathbf{E}$. By Convention
\ref{arrangementconvention}, we denote $\mathcal{V}_E$ also the underlying canonical
poset in $\Z^n$. The first step of the algorithm described
in subsection \ref{resolutions} yields the following short exact sequence:
\begin{equation}\label{firststep}
0 \longrightarrow E^1 \longrightarrow F \longrightarrow E \longrightarrow 0,
\end{equation}
where $F \cong \bigoplus_{\uc \in \mathcal{V}_E} S(-\uc)^{n_\uc}$ is a free
cover of $E$ and  $E^1$ the first syzygy of $E$. By
construction, $\mathcal{V}_E$ is also admissible for $E^1$ and $F$. On
the other hand, $E^1$ and $F$ both are reflexive and therefore come
with their own canonical admissible posets $\mathcal{V}_{E^1}$ and
$\mathcal{V}_F$, respectively, corresponding to vector space arrangements
in the limit vector spaces $\mathbf{E^1}$ and $\mathbf{F}$. From the
arguments of Proposition
\ref{globaldimbound} we conclude that $\mathcal{V}_{E^1}$ is a proper
subposet of $\mathcal{V}_E$. For $\mathcal{V}_F$ we get:

\begin{proposition}\label{coordinatearrangement}
With above notation the posets $\mathcal{V}_F$ and $\mathcal{V}_E$ coincide
as subsets of $\Z^n$.
\end{proposition}

\begin{proof}
By construction, we have
\begin{equation*}
F_\uc = \bigoplus_{\substack{\uc' \in \mathcal{V}_E \\ \uc' \leq \uc}} S(-\uc')^{n_{\uc'}}_\uc,
\quad \text{ and } \quad \dim F_\uc = \sum_{\substack{\uc' \in \mathcal{V}_E \\ \uc' \leq \uc}}
n_{\uc'}.
\end{equation*}
In particular, $F_\uc = F_{\uc'}$, where $\uc' = \max\{\uc'' \in \mathcal{V}_E \mid \uc'' \leq
\uc\}$. Therefore, $\mathcal{V}_F \subseteq \mathcal{V}_E$. To show equality, it suffices
to check that $F_{\uc'} \neq F_{\uc}$ whenever $\uc' <  \uc \in \mathcal{V}_E$. As in
the proof of Proposition \ref{reflexivelcmlattice}, $\uc' < \uc$ implies that
there exists some $k \in \on$ such that $E_\uc \subseteq E^k(c_k)$ but
$E_\uc \nsubseteq E^k(c_k')$ and $E_\uc' \subseteq E^k(c_k')$, where
$\uc = (c_1, \dots, c_n)$, $\uc' = (c_1', \dots, c_n')$. In particular,
$E_{\uc'} \subsetneq E_\uc$. But then there must exist at least one
$\uc'' \leq \uc$, but $\uc'' \not \leq \uc'$ such that $n_{\uc''} > 0$.
Then $\dim F_\uc > \dim F_{\uc'}$ and the claim follows.
\end{proof}

Another way of phrasing Proposition \ref{coordinatearrangement} is that we have
constructed a surjection of the vector space arrangement $\mathcal{V}_E$ by a
combinatorially equivalent {\em coordinate space arrangement} $\mathcal{V}_F$
such that any $F_\uc \in \mathcal{V}_F$ maps surjectively onto $E_\uc \in \mathcal{V}_E$.
The vector space arrangement $\mathcal{V}_{E^1}$ then can be considered as a
{\em syzygy arrangement} of $\mathcal{V}_E$. By iterating, our free resolution
of $E$ then yields a {\em resolution of $\mathcal{V}_E$ in terms of
coordinate vector space arrangements}:
\begin{equation*}
0 \longrightarrow \mathcal{V}_{F_t} \longrightarrow \cdots \longrightarrow
\mathcal{V}_{F_0} \longrightarrow \mathcal{V}_E \longrightarrow 0,
\end{equation*}
where every $\mathcal{V}_{F_i}$ is a coordinate vector space arrangement which is
combinatorially equivalent to the $i$-th syzygy arrangement $\mathcal{V}_{E_i}$.
This is a remarkable observation which yields a new class of invariants of
vector space arrangements which can be expressed in combinatorial terms, where
the vector space dimension in the arrangements $\mathcal{V}_{F_i}$ play
the role of Betti numbers. In section \ref{hyperplanearrangements} we will determine these
invariants for the case of central hyperplane arrangements.
However, so far it is not clear whether these have any use
in the study of general vector space arrangements. From this discussion and Proposition
\ref{coordinatearrangement} we conclude:

\begin{theorem}\label{invariants}
Let $E$ be a $\Z^n$-graded, finitely generated, reflexive $S$-module. Then the poset of
nonzero graded Betti numbers is determined by the embedding of the poset given by the
underlying vector space arrangement $\mathcal{V}_E$ into $\Z^n$
For given $X \in \mathcal{V}_E$, the corresponding Betti number depends only on
$\mathcal{V}_E$.
\end{theorem}

By the following construction our correspondence between
reflexive $S$-modules and vector space arrangements at least provides a method
to efficiently compute resolutions of arrangements (see Remark \ref{vecresremark}).

\begin{definition}\label{reflexivemodeldef}
Let $\mathcal{V}$ be a finite vector space arrangement in some finite-dimensional
vector space $V$. We say that a reflexive module $E$ in $\zsmod$ for some $n$ is
a  {\em reflexive model} for $\mathcal{V}$, if there exists an isomorphism
of vector spaces $\mathbf{E} \rightarrow V$ which induces an isomorphism of
vector space arrangements $\mathcal{V}_E \cong \mathcal{V}$.
\end{definition}

\begin{proposition}\label{reflexivemodelexists}
Every finite vector space arrangement has a reflexive model.
\end{proposition}

\begin{proof}
Let $\mathcal{V}$ be a finite vector space arrangements in some vector space $V$.
We can assume without loss of generality that $\mathcal{V}$ is nontrivial and
enumerate its elements which are different from $V$ and $0$ by $X_1, \dots, X_n$.
Then we define a set of filtrations
$V^k(i)$ of $V$ with
\begin{equation*}
V^k(i) =
\begin{cases}
0 & \text{ for } i < 0 \\
X_k & \text{ for } i = 0 \\
V & \text{ for } i > 0.
\end{cases}
\end{equation*}
By Theorem \ref{klythm} this data gives rise to a reflexive
module in $\zsmod$. Its underlying vector space arrangement by construction
coincides with $\mathcal{V}$.
\end{proof}

Of course, the choice of a reflexive model as in the proof of Proposition
\ref{reflexivemodelexists} is not canonical and far from being unique. For
instance, for a given arrangement one could instead choose a minimal collection
of flags whose intersection poset generate $\mathcal{V}$.

\begin{remark}
Note that our resolutions only depend on the actual arrangement. If, say,
some vector space is contained in more than one filtration $E^k(i)$
associated to some reflexive module $E$, the resolution of the underlying
vector space arrangement by coordinate
arrangements does not depend on this multiplicity.
However, the actual embedding of $\mathcal{V}_E$ in $\Z^n$ (and therefore
the degrees which show up in the minimal resolution) does depend
on possible multiplicities (see Example \ref{einfach} below).
\end{remark}

\begin{remark}\label{vecresremark}
Let a vector space arrangement $\mathcal{V}$ and a reflexive model $E_\mathcal{V}$
be given. By Proposition \ref{filtaux} (\ref{filtauxii}) we can choose an embedding
of $E_\mathcal{V}$ into some free module $S(-\uc)^{\rk E_\mathcal{V}}$ where
$E^k(c_k) = 0$ for every $k \in \on$. If
we are able to determine the first step of the resolution as in sequence
(\ref{firststep}), then we obtain a representation of $E_\mathcal{V}$ as the image of
a monomial matrix $F_0 \rightarrow S(-\uc)^{\rk E_\mathcal{V}}$. Then by computing
a resolution of $E_\mathcal{V}$, say with help of a computer algebra system, we obtain
also a resolution $\mathcal{V}$.
\end{remark}

\begin{example}\label{einfach}
Let $\mathcal{V}$ be a vector space arrangement in $V \cong \K^r$ generated by a
family of subvector spaces $\{V_1, \dots, V_t\}$
which has only trivial intersections, i.e. $V_i \cap V_j = \{0\}$ whenever $i \neq j$.
We assume that the $V_i$ span $V$. Denote $I_1 \sqcup \cdots \sqcup I_t = \on$ some
decomposition. If $E$ is given by filtrations
\begin{equation*}
V^k(i) =
\begin{cases}
0 & \text{ for } i < i_1^k \\
V_l & \text{ for } i_1^k \leq i < i_2^k \text{ and } k \in I_l\\
V & \text{ for } i_2^k \leq i
\end{cases}
\end{equation*}
and integers $i_1^k < i_2^k$ and $k \in \on$,
then we get as a minimal free resolution:
\begin{equation*}
0 \longrightarrow S(-\underline{i}_2)^{\sum_l \dim V_l - r} \longrightarrow
\bigoplus_{l = 1}^t S(-\uc_l)^{\dim V_l} \longrightarrow E \longrightarrow 0.
\end{equation*}
Here, $\underline{i}_2 = (i_2^1, \dots, i_2^n)$ and
$\uc_l = (c_l^1, \dots, c_l^n)$ with $c_l^k = i_1^k$ if $k \in I_l$ and $c_l^k = i_2^k$ else.
\end{example}

\begin{example}[see also \cite{Brenner08} \S 6 \& \S7]]
Let $I \subset S$ be a monomial ideal and denote $T \subseteq \Z^n$ the minimal set of
generators of
its $\lcm$-lattice in the sense of Example \ref{admissibleexample3}. The first term of
a minimal resolution of $I$ is given by
\begin{equation*}
0 \longrightarrow I^1 \longrightarrow \bigoplus_{\uc \in T} S(-\uc) \longrightarrow I
\longrightarrow 0
\end{equation*}
such that the first syzygy $I^1$ is a reflexive $S$-module. Consider the associated
short exact sequence of limit vector spaces $0 \rightarrow \mathbf{I^1} \rightarrow
\mathbf{F} \rightarrow \mathbf{I} \rightarrow 0$, where $\dim \mathbf{I} = 1$ and
$\mathbf{F} \cong \bigoplus_{\uc \in T} \mathbf{F}_\uc$ with $\dim \mathbf{F}_\uc = 1$
for all $\uc \in T$. The filtrations of $\mathbf{F}$ are given by
\begin{equation*}
F^k(i) = \bigoplus_{\uc \in T,\ c_k \leq i}\mathbf{F}_\uc.
\end{equation*}
The filtrations of $\mathbf{I^1}$ then are given by the kernels of the homomorphisms
$F^k(i) \rightarrow \mathbf{I}$. Betti numbers of monomial ideals are not purely
combinatorial, as in general, for instance, they depend on the characteristic of $\K$.
It would be of some interest to study this phenomenon in terms of vector space
arrangements.
\end{example}

The approach presented here reduces the problem of constructing free resolutions to a
straightforward linear algebra problem which allows to construct graded Betti numbers
by iteratively analyzing linear dependencies. In general, we cannot expect to obtain
closed forms for such resolutions apart from very easy cases, such as in example
\ref{einfach}. However, as we will see in section \ref{hyperplanearrangements}, our
approach is powerful enough to yield closed expressions for the case of hyperplane
arrangements.

\subsection{Duality of resolutions and local cohomology}\label{duality}

We have seen so far that by using appropriate $\lcm$- and $\gcd$-lattices,
projective and injective resolutions of combinatorially finite
$S$-modules  both can be constructed straightforwardly.
In this subsection we will show that it actually suffices to consider only
projective resolutions. More precisely, given some combinatorially finite
$S$-module $E$, we will show that we from a given minimal projective
resolution of $E$ a minimal injective resolution can directly be constructed.

For this consider the $\Z^n$-graded Grothendieck-Cousin complex of $S$, which
provides a minimal injective resolution $0\rightarrow S \rightarrow I_S$ in $\zsmod$,
with $I_S$ explicitly given by (see subsection \ref{divisorialmodules}):
\begin{equation*}
0 \longrightarrow \dual{S}_\on(\one) \longrightarrow \cdots \longrightarrow
\bigoplus_{I \subseteq \on, \ \vert I \vert = n - i} \dual{S}_I(\one)
\longrightarrow \cdots \longrightarrow \dual{S}(\one) \longrightarrow 0,
\end{equation*}
where $\one = (1, \dots, 1) \in \Z^n$.

Now assume that we have a projective resolution of $E$, that is, a finite complex
of projective modules $P_E$ in $\zsmod$ which is quasi-isomorphic to $E$, where
$E$ is considered as a complex concentrated in degree $0$. Then we have the
following chain of quasi-isomorphisms
\begin{equation*}
E \cong P_E \cong P_E \otimes_S S \cong P_E \otimes I_S,
\end{equation*}
where, due to Lemma \ref{injaux} (\ref{injauxi}), the last complex is a complex
of injective modules in $\zsmod$. This can be expressed more generally in the
setting of derived categories as follows:

\begin{lemma}\label{homotopylemma}
Denote $K^b(\text{$\Z^n$-$S$-$\operatorname{proj}$})$ and
$K^b(\text{$\Z^n$-$S$-$\operatorname{inj}$})$ the homotopy categories of bounded
projective and injective complexes, respectively, in $\zsmod$, and denote
$D^b(\zsmod)$ the bounded derived category of $\zsmod$. Then the functors
\begin{equation*}
- \otimes_S I_S :
K^b(\text{$\Z^n$-$S$-$\operatorname{proj}$}) \rightarrow
K^b(\text{$\Z^n$-$S$-$\operatorname{inj}$}) \text{ and } \Hom(I_S, -) :
K^b(\text{$\Z^n$-$S$-$\operatorname{inj}$}) \rightarrow
K^b(\text{$\Z^n$-$S$-$\operatorname{proj}$})
\end{equation*}
are well-defined and fit as
mutually inverse functors into the
following commutative diagram of equivalences of triangulated categories:
\begin{equation*}
\xymatrix{
K^b(\text{$\Z^n$-$S$-$\operatorname{proj}$}) \ar[rr]^{- \otimes_S I_S}_\cong \ar[d]_\cong
& & K^b(\text{$\Z^n$-$S$-$\operatorname{inj}$}) \ar[d]_\cong \ar[rr]^{\Hom(I_S, -)}_\cong
& & K^b(\text{$\Z^n$-$S$-$\operatorname{proj}$}) \ar[d]_\cong \\
D^b(\zsmod) \ar[rr]^{- \otimes_S S}_\cong & & D^b(\zsmod) \ar[rr]^{Hom(S, -)}_\cong
& & D^b(\zsmod)
}
\end{equation*}
\end{lemma}

\begin{proof}
The vertical isomorphisms are the standard isomorphisms. It follows from Lemma
\ref{injaux} (\ref{injauxi}) that the tensor product of a complex of projective
modules yields a complex of injective modules. Likewise, it follows from
\ref{injaux} (\ref{injauxii}) that $\Hom_S(I_S, -)$ applied to a complex of injective
modules yields a complex of  projective modules. It is then clear that the
horizontal arrows yield equivalences and that the diagram commutes.
\end{proof}

For simplicity, we will assume from now that $E$ is a finitely generated module.
Then a minimal projective resolution $P_E$ automatically is a free resolution.
The complex $P_E \otimes_S I_S$ is not a resolution of $E$ in the traditional
sense, as it is nonzero in degrees smaller than $0$.
Our aim is to obtain from $P_E \otimes_S I_S$ a minimal
injective resolution for $E$. To this end, consider the decomposition
\begin{equation*}
I_S = \bigoplus_{i = 0}^n \bigoplus_{I \subseteq \on, \vert I \vert = n -i} \dual{S}_I
(e_I)[i - n].
\end{equation*}
Note that this decomposition is not a proper decomposition of $I_S$ into
subcomplexes, i.e. we disregard the differential of $I_S$ for the moment.
Then we can decompose
\begin{equation*}
P_E \otimes_S I_S =
\bigoplus_{i = 0}^n \bigoplus_{I \subseteq \on, \vert I \vert = n -i}
P_E \otimes_S \dual{S}_I (\one)[i - n].
\end{equation*}

Explicitly, if we write $P_{E, i} \rightarrow P_{E, i - 1}$ as
$\bigoplus_j S(\uc_{i, j}) \overset{\phi_i}{\rightarrow} \bigoplus_k S(\uc_{i - 1, j})$,
we get by Lemma \ref{injaux} (\ref{injauxi}) for $P_{E} \otimes \dual{S}_I(\one)$:
\begin{equation*}
\bigoplus_k \dual{S}_I(\uc_{i, k} + \one) \longrightarrow \bigoplus_j
\dual{S}_I(\uc_{i - 1, j} + \one)
\end{equation*}
Now, applying $\Hom_S(S_I, -)$, we get by Lemma \ref{injaux} (\ref{injauxii}):
\begin{equation*}
\bigoplus_k S_I(\uc_{i, k} + \one) \overset{\phi_i}{\longrightarrow} \bigoplus_j
S_I(\uc_{i - 1, j} + \one).
\end{equation*}
That is, we can naturally identify the complex $\Hom_S(S_I, P_E \otimes_S \dual{S}_I(\one))$
with the localization of the complex $P_E$ at the monomial $x^I \in S$. The complex
$(P_E)_{x^I}$ then is a projective resolution of $E_{x^I}$ over $S_{x^I}$ which, in
general, is no longer minimal. In the following
theorem we will use this information and apply Lemma \ref{minimizeinjres} to construct
from $P_E \otimes_S I_S$ a minimal injective resolution for $E$.

\begin{theorem}\label{dualitytheorem}
Let $E$ be a finitely generated $S$-module and $P_E$ a minimal free resolution
of $E$. Then, by splitting off summands from $P_E \otimes_S I_S$, we obtain a minimal
injective resolution $I_E$ of $E$. Its components with respect to the injectives
$S_I$ are given by
\begin{equation*}
P_{E_{x^I}} \otimes_S \dual{S}_I(\one),
\end{equation*}
where $P_{E_{x^I}}$ denotes the minimal free resolution of $E_{x^I}$ over $S_{x^I}$.
\end{theorem}

\begin{proof}
First observe that $P_E \otimes_S \dual{S}_I(\one) \cong (P_E)_{x^I} \otimes_S
\dual{S}_I(\one)$ for every $I \subseteq \on$.
Consider the decomposition of $P_E \otimes_S I_S$ as above. Then, for every
$I \subseteq \on$, for every summand of the projective resolution $(P_E)_{x^I}$
which can be split of, we can can split of the corresponding summand of
$P_E \otimes_S I_S$ by Lemmas \ref{minimizeprojres} and \ref{minimizeinjres}.
Therefore, reducing $(P_E)_{x^I}$ to a minimal resolution $P_{E_{x^I}}$
over $S_{x^I}$ for every $I \subseteq \on$ simultaneously reduces $P_E \otimes_S I_S$
to a minimal injective resolution of $E$ with the stated properties.
\end{proof}

Using Theorem \ref{dualitytheorem}, we can rederive a well-known correspondence
between graded Betti and Bass numbers of finitely generated modules.

\begin{corollary}[see \cite{Miller00} \S 5]\label{bettibasscorrespondence}
Let $E$ be a finitely-generated $\Z^n$-graded $S$-module and for $I \subset \on$ denote
$\tau_I$ the corresponding face of the positive orthant in $\R^n$. Then the graded Bass
number $b^i(\tau_I, \uc)$ equals the graded Betti number $\beta_{p_I - i}(\uc + \one)$ of
the $S_{x^I}$-module $E_{x^I}$, where $p_I$ denotes the projective dimension of
$E_{x^I}$ as an $S_{x^I}$-module.
\end{corollary}

\begin{remark}
As already mentioned, Corollary \ref{bettibasscorrespondence} is not a new result
but has been proved in (essentially) the same generality in \cite{Miller00}. However,
in loc. cit. the proof involves several technical ad-hoc methods, whereas by
Lemma \ref{homotopylemma} this correspondence is just the specialization
of a  natural lift of the trivial autoequivalence of $D^b(\zsmod)$ to the level of homotopy
categories. Also, we can construct the full minimal injective resolution of a given
module $E$, whereas in loc. cit., starting with the
case of modules of finite length, an inductive argument
for the pieces $P_E \otimes_S \dual{S}_I(\one)[i - n]$ is given.
Moreover, we can interpret now Alexander duality as studied in \cite{Miller00} in more
general terms by the autoequivalence of $D^b(\zsmod)$
given by the Matlis duality functor.
\end{remark}

\begin{remark}
The graded Grothendieck-Cousin complex is a very simplified version of
the general Grothendieck-Cousin complex for local rings
(see \cite{Sharp69}). However, the construction of Lemma
\ref{homotopylemma} might be of more general
interest, e.g. for studying Bass numbers and local cohomologies of
modules over regular local rings.
\end{remark}

Another application is the following formula for
computing the local cohomologies of modules in $\zsmod$:

\begin{theorem}\label{loccohom1}
Let $V \subseteq \mathbf{A}^n_K$ be a torus invariant closed subscheme
and denote $R \Gamma_V : D^b(\zsmod) \rightarrow
D^b(\zsmod)$ the right derived local cohomology functor. Then for any object
$E$ in $D^b(\zsmod)$ which has a projective representative $P_E$ in
$K^b(\text{$\Z^n$-$S$-$\operatorname{proj}$})$ we get the following formula
for local cohomology:
\begin{equation*}
H^i_V(E) \cong H^i(P_E \otimes_S R \Gamma_V S).
\end{equation*}
\end{theorem}

\begin{proof}
As in Lemma \ref{homotopylemma}, we represent $E$ by a complex of injectives given
by $P_E \otimes_S I_S$. Then clearly $R \Gamma_V(E) \cong \Gamma_V (P_E \otimes_S I_S)$.
Then, decomposing $P_E \otimes_S I_E$ as before, we get by Lemma
\ref{acycprop} (\ref{acycpropi}) that
$\Gamma_V\big(P_E \otimes_S \dual{S}_I(\one)[i - n]\big) = 0$ iff the support
of $\dual{S}_I(e_I)$ is not contained in $V$. Hence, the surviving part of
$\Gamma_V\big(P_E \otimes_S \dual{S}_I(\one)[i - n]\big)$ equals
$P_E \otimes_S R\Gamma_V S$.
\end{proof}

\begin{remark}
Theorem \ref{loccohom1} is another variant of a well-known formula, e.g. see
\cite{brunsherzog} Thm. 3.5.6 and \cite{Miller00} Thm. 6.2.
\end{remark}

\subsection{Computing local cohomologies}\label{loccohomcomp}

Given a finitely generated module $E$ in $\zsmod$, Theorem \ref{dualitytheorem} tells
us how to construct an admissible $\gcd$-lattice. By Proposition \ref{localequiv},
for every $I \subset \on$, the
category of $\Z^n$-graded $S_{x^I}$-modules is equivalent to the category of
$\Z^{\on \setminus I}$-graded $\big(S / (x_i \mid i \in I)\big)$-modules. Hence,
for every $E_{x^I}$ we have an admissible $\lcm$-lattice $\cL_I$ in
$\Zbar^{\on \setminus I}$.
We embed these $\cL_I$ into $\Zbar^n$ by $\uc = (c_i \mid i \in \on \setminus I)
\mapsto \iota_I(\uc) = (c'_1, \dots, c'_n) \in \Zbar^n$, where
$c'_i = \begin{cases}\infty & \text{ if } i \in I\\ c_i - 1 & \text{ otherwise}.\end{cases}$
Then by Theorem \ref{dualitytheorem} $\bigcup_{I \subset \on} \iota_I(\cL_I)$
is a superset of all degrees
of objects contained in the minimal injective resolution of $E$. Hence,
the $\gcd$-lattice $\cG$ generated by all $\iota_I(\cL_I)$ in $\Zbar^n$ is $E$-admissible
in the sense of Definition \ref{gcdadmissibledef}.

Now assume that we have constructed some minimal injective resolution $I_E$ and
let $V \subseteq \A^n_\K$ be any closed torus invariant subscheme. Then $H^i_V E$
is determined completely once its graded parts $(H^i_V E)_\uc = H^i(\Gamma_V I_E)_\uc$
have been determined for every $\uc \in \cG$. For every $i \geq 0$ we decompose
$(\Gamma_V I_E)^i =: A^{i, \uc} \oplus B^{i, \uc}$ such that the first summand
contains all summands $\dual{S}_I(\underline{d})$ with $d_i \geq c_i$ for every
$i \in \on \setminus I$. Now, for any two summands $S_J(\underline{d})$ and
$S_I(\underline{d}')$ of $(\Gamma_V I_E)^i$ and $(\Gamma_V I_E)^{i + 1}$,
respectively, we know by Lemma \ref{injaux} (\ref{injauxii}) that the corresponding
entry of the representing matrix of the differential
$\phi^i: (\Gamma_V I_E)^i \rightarrow (\Gamma_V I_E)^{i + 1}$ (see Lemma
\ref{minimizeinjres}) is zero unless
$I \subseteq J$ and $d_i \leq d_i'$ for every $i \in \on \setminus J$. In particular,
the image of restriction $\phi^i\vert_{B^{i, \uc}}$ is contained in $B^{i + 1, \uc}$.
The following proposition will be useful.

\begin{proposition}\label{cohomologyshift}
With above notation we have  $(H^i_V E)_\uc \cong \big(H^i(A^{\bullet, \uc})\big)_\uc
\cong \big(H^{i + 1}(B^{\bullet, \uc})\big)_\uc$ for every $\uc \in \cG$
\end{proposition}

\begin{proof}
By the above discussion
the differential $\phi^i : A^i \oplus B^i \rightarrow A^{i + 1} \oplus B^{i + 1}$ decomposes
into blocks as follows
\begin{equation*}
\phi^i =
\begin{pmatrix}
\phi^i_A & 0 \\
\phi^i_{AB} & \phi^i_B.
\end{pmatrix}
\end{equation*}
with $\phi^i_A : A^i \rightarrow A^{i + 1}$, $\phi^i_{AB} : A^i \rightarrow B^{i + 1}$,
and $\phi^i_B : B^i \rightarrow B^{i + 1}$. Then we have complexes $\dots \rightarrow A^i
\overset{\phi^i_A}{\rightarrow} A^{i + 1} \rightarrow \cdots$ and
$\dots \rightarrow B^i \overset{\phi^i_B}{\rightarrow} B^{i + 1} \rightarrow \cdots$ which
fit into a short exact sequence of complexes
\begin{equation*}
0 \longrightarrow B^\bullet \longrightarrow A^\bullet \oplus B^\bullet \longrightarrow
A^\bullet \longrightarrow 0.
\end{equation*}
Then the assertion follows from the induced long cohomology sequence and the
fact that the complex $A^\bullet \oplus B^\bullet$ is exact.
\end{proof}

Below we will mostly be interested in the local cohomologies $H^i_x E$ with respect
to the torus fixed point $x$ in $\A^n_\K$ (or the maximal $\Z^n$-graded ideal of $S$, respectively). If a minimal free
resolution $F_\bullet \rightarrow E \rightarrow 0$ is given, then by Theorem
\ref{loccohom1}, $H^i_x E_\mH$ coincides with the $i$-th cohomology
of the complex $F_\bullet \otimes_S \dual{S}(\one)[-n]$. This implies
that the ``interesting'' degrees are determined by the $\gcd$-lattice $\cG$
generated by the degrees of the nonzero graded Bass numbers $b^i(\uc)$ of $E$, i.e. by
$\{\uc - \one \mid \uc \in \cL_\emptyset\}$.

\begin{definition}\label{adjacentdefinition}
We call $\uc \in \Z^n$ {\em adjacent} to $\ud \in \cG$ if $\uc \leq \ud$
and $\uc \nleq \ud'$ for all $\ud' \in \cG$ with $\ud \nleq \ud'$.
\end{definition}

For any $\uc \in \Z^n$ we have $(H^i_x E)_\uc = 0$ if $\uc$ is not adjacent
to any $\ud \in \cG$ and $(H^i_x E)_\uc = (H^i_x E)_\ud$ if $\uc$ is
adjacent to $\ud$.

\begin{example}\label{einfachbass}
Consider vector space arrangements with reflexive model $E$ as in Example \ref{einfach}.
The canonical admissible $\lcm$-lattice is given by $\underline{i}_2$ and the
$\uc_k$. Then $\cG$ is given by
$\ud_P := \gcd\{\uc_l \mid l \in P\} -
\one$ for every subset $P$ of $\{1, \dots, t\}$.
Explicitly, $\ud_P = (d_P^1, \dots, d_P^n)$, where $d_P^k = i_1^k - 1$ if $k \in \bigcup_{p \in P}
I_p$ and $d_P^k = i_2^k - 1$ otherwise. The only possibly nontrivial local cohomologies
are given by $(H^{n - 1}_x E)_\uc$ and $(H^n_x E)_\uc$, where $\uc$ is adjacent to some
$\ud_p$. Then $\dim (H^{n - 1}_x E)_{\ud_P}$ and $\dim (H^n_x E)_{\ud_P}$
coincide with the dimension of the cohomologies of the complex
\begin{equation*}
0 \longrightarrow V^1 \longrightarrow \bigoplus_{p \in P} V_p \longrightarrow 0,
\end{equation*}
which is concentrated in degrees $n - 1$ and $n$, respectively. However, it is
more convenient to read off the dimensions from the $n$-th and $(n + 1)$-st
cohomologies of the following complex which is concentrated in degrees $n$
and $n + 1$:
\begin{equation*}
0 \longrightarrow \bigoplus_{p \notin P} V_p \longrightarrow V \longrightarrow 0.
\end{equation*}
From this we read off
\begin{equation*}
\dim (H^{n - 1}_x E)_{\ud_P} = \sum_{i \notin P} \dim V_i - \dim \sum_{i \notin P} V_i
\quad \text{ and } \quad
\dim (H^n_x E)_{\ud_P} = \dim V - \dim \sum_{i \notin P} V_i.
\end{equation*}
for every $P \subseteq \{1, \dots, t\}$.
\end{example}

\section{Hyperplane arrangements}\label{hyperplanearrangements}

In this section we want to apply our machinery to a special class of reflexive
modules, whose associated filtrations form a hyperplane arrangement. These modules
are introduced in subsection \ref{hyperplanemodules}. In subsection
\ref{combinatorics} we introduce some notions from combinatorics and characterize
some formulas in a form which is suitable for our applications. For a general
reference for hyperplane arrangements and their associated combinatorics we refer
to \cite{OrlikTerao}. In subsection \ref{resolutionofhyperplanemodules} construct
free resolutions of hyperplane modules and determine their $\Z^n$-graded Betti
numbers. In particular, we will show that these numbers are completely determined
by the combinatorics of the associated hyperplane arrangement.
Using these results and the duality of subsection \ref{duality}, we will
in subsection \ref{hyperplanesloccohom} completely
determine the local cohomology of hyperplane modules with respect to the maximal
$\Z^n$-graded ideal of $S$.

\subsection{Hyperplane modules}\label{hyperplanemodules}

A reflexive module $E$ in $\zsmod$ is called a {\em hyperplane module} if it is
a reflexive model of a finite hyperplane arrangement in the sense of Definition
\ref{reflexivemodeldef}. Let $\mathcal{H}$ be a hyperplane arrangement in
some vector space $V \cong \K^r$, generated
by hyperplanes $H_1, \dots, H_t$.
For simplicity, we will assume that there
are  $\{k_1, \dots, k_t\} \subseteq \on$
such that the associated filtrations are of the form
\begin{equation*}
E^{k_l}(i) =
\begin{cases}
0 & \text{ for } i < i_{k_l}\\
H_l & \text{ for } i_{k_l} \leq i < j_{k_l} \\
V & \text{ for } j_{k_l} \leq i,
\end{cases}
\quad and \quad
E^k(i) =
\begin{cases}
0 & \text{ for } i < j_k\\
V & \text{ for } j_k \leq i.
\end{cases}
\text{ for } k \notin \{k_1, \dots, k_t\}
\end{equation*}
for integers $i_{k_j} < j_{k_j}$ and $j_k$. In other words,
every hyperplane is contained in precisely one filtration and  we will allow for
additional coordinates with trivial filtrations.
The associated canonical admissible $\lcm$-lattice $\cL_\mathcal{H}$ in $\Z^n$
of $E_\mathcal{H}$ is given
by $\cL_\mathcal{H} = \{\uc^X = (c^X_1, \dots, c^X_n) \in \Z^n
\mid X \in \mathcal{H}\}$, where
\begin{equation*}
c^X_k =
\begin{cases}
i_{k_j} & \text{ if } k = k_j \text { for some } j \text{ and } X \subseteq H_j\\
j_k & \text{ else}.
\end{cases}
\end{equation*}

Recall that $\mH$ is called {\em essential} if $\bigcap_i H_i = \{0\}$. If
$\bigcap_i H_i = C$ for some nontrivial subvector space $C$ of $V$, we can apparently
split $V \cong V' \oplus C$ and $H_i \cong H_i' \oplus C$, where $V' := V / C$ and
$H_i' := H_i / C$ for every $i$. Obviously, the $H_i'$ generate a hyperplane arrangement
$\mH'$ in $V'$. We observe that the isomorphism $V \cong V' \oplus C$ induces an
isomorphism of hyperplane arrangements between $\mH$ and $\mH' \oplus C$.
for $\mathcal{H}$ splits into a direct sum
$E_{\mathcal{H}'} \oplus S(-j_1, \dots, -j_n)^{\dim C}$, where $E_{\mathcal{H}'}$
is a reflexive model for the hyperplane arrangement $\mathcal{H}' = \{H_1', \dots,
H_t'\}$. So the minimal free resolutions for $E_\mH$ and $E_{\mH'}$ coincide except
for the 0-th term, where for $E_\mH$ we get $S(-j_1, \dots, -j_n)^{\dim C}$ as an
extra summand. So below it will be safe to assume that $\mH$ is essential.

\subsection{On the combinatorics of hyperplane arrangements}\label{combinatorics}

We assume that $\mH$ is an essential hyperplane arrangement.
We consider $\mH$ as a ranked poset with respect to the dimension.

{\bf Caution:} Note that our this ordering on $\mH$ is opposite to the standard
ordering usually considered in the context of hyperplane arrangements see e.g. \cite{OrlikTerao}. This might be somewhat confusing for some readers, but it
is the more natural order with respect to our general setup.

For any subset
$\p$ of $\mathcal{H}$, the {\em M\"obius function}
of $\p$ is defined as
\begin{equation*}
\mu^\p(X, Y) =
\begin{cases}
1 & \text{ if } X = Y\\
-\sum_{X \leq Z < Y, \Z \in \p} \mu^\p(X, Z) & \text{ if } X \leq Y\\
0 & \text{ else},
\end{cases}
\end{equation*}
where we write $\mu$ for $\mu^\mH$. Using the reduced dimension
and the M\"obius function, we take the following definitions from \cite{GreeneZaslavsky}.

\begin{definition}
Let $\p$ be any subset of $\mH$.
\begin{enumerate}[(i)]
\item For $i \leq j \leq r$ the
{\em Whitney numbers} are defined as
\begin{equation*}
w_{ij}^\p = \sum_{\substack{X \leq Y \in \p\\ \dim X = i,\ \dim Y = j}}\mu^\p(X, Y).
\end{equation*}
We write $w_{ij}^\mH = w_{ij}$ and $w_{ij}^{\{Y \leq X\}} = w^X_{ij}$ for
some $X \in \mH$.

\item The {\em beta invariants} are defined as
\begin{equation*}
\beta^\p = (-1)^D \sum_{d = 0}^D d w_{dD}^\p,
\end{equation*}
where $D = \max\{\dim X \mid X \in \p\}$.

\item For particular posets of the form $\p_{k, X} := \{Y \leq X \mid \dim Y \leq k\}
\cup \{ X \}$
we write $\mu^{k, X} := \mu^{\p_{k, X}}$, $w_{ij}^{k, X} := w_{ij}^{\p_{k, X}}$,
$\beta^{k, X} := \beta^{\p_{k, X}}$, and $\beta^X := \beta^{\dim X - 1, X}$.
\end{enumerate}
\end{definition}

For our constructions we introduce now another combinatorial characterization of
beta invariants in terms of path length counting formulas.

\begin{definition}\label{chaindef}
Let $\p$ be any subset of $\mH$ and denote $c$ any chain $X_l < X_{l - 1} < \cdots <
X_0 = X$ in $\p$. Then we set $l(c) := l$, i.e. the length of $c$ minus one, and
$\dim c := \dim X_l$.
For any $X \in \mH$ we denote $C_{k, X}$ the set of chains in $\p_{k, X}$.
For any $Y \in \p_{k, X}$ we denote $C_{k, Y, X} \subseteq C_{k, X}$ the subset of
chains starting at $Y$.
\end{definition}

With these notations we get the following lemma.

\begin{lemma}\label{betachainlemma}
Let $k \geq 0$ and $X \in \mH$. Then
\begin{equation*}
\beta^{k, X} = (-1)^k \sum_{c \in C_{k, X}}(-1)^{l(c)} \dim c.
\end{equation*}
\end{lemma}

\begin{proof}
We write
\begin{equation*}
\sum_{c \in C_{k, X}}(-1)^{l(c)} \dim c = \sum_{d = 0}^{\dim X} d
\sum_{Y \leq X, \dim Y = d} \sum_{c \in C_{k, Y, X}} (-1)^{l(c)}.
\end{equation*}
Using the formula from \cite{OrlikTerao}, Prop. 2.37, we get:
\begin{align*}
(-1)^k \sum_{c \in C_{k, X}}(-1)^{l(c)} \dim c & = (-1)^k \sum_{d = 0}^{\dim X} d
\sum_{Y \in \p_{k, X}, \dim Y = d} \mu^{k, X}(Y, X)\\
& = (-1)^k \sum_{d = 0}^{\dim X} d w^{k, X}_{d\dim X} = \beta^{k, X}.
\end{align*}
\end{proof}

\subsection{Free resolutions}\label{resolutionofhyperplanemodules}

Let $\mH$ be a hyperplane arrangement. In this subsection we will a
construct an explicit minimal resolution of $\mH$ by coordinate arrangements
$0 \rightarrow \mathcal{F}_s
\overset{\phi_s}{\rightarrow} \cdots \overset{\phi_1}{\rightarrow} \mathcal{F}_0
\overset{\phi_0}{\rightarrow} \mathcal{H} \rightarrow 0$.
By results of subsection \ref{resolutions}, for a hyperplane module $E_\mH$
we get $\mathcal{F}_i \cong \bigoplus_{X \in \mathcal{H}} \K^{\beta_i(\uc^X)}$,
where $\beta_i(\uc^X)$ denote the graded Betti numbers of $E_\mH$. In particular,
for $\beta_i(\uc)$ to be nonzero, $\uc$ necessarily equals some $\uc^X$ for
$X \in \mathcal{H}$.
By definition, for $i > 0$, the $i$-th syzygy arrangement $\mathcal{H}^i$
consists of the arrangement of kernels of $\phi_{i - 1}$ with the convention
that $\mathcal{H} = \mathcal{H}^0$ is considered as the 0-th syzygy.
In particular, for
any $i \geq 0$ we have a short exact sequence of vector space arrangements
$0 \rightarrow \mathcal{H}^{i + 1} \rightarrow \mathcal{F}_i \rightarrow
\mathcal{H}^i \rightarrow 0$. As every $\mathcal{F}_i$ and every $\mathcal{H}^i$
can be considered as a $\K$-linear representation of the intersection poset of
$\mathcal{H}$, we get for any $X \in \mathcal{H}$ a series of vector spaces
$X = X^0, X^1, \dots$ and $F_0^X, F_1^X, \dots$ in $\mathcal{H}^0, \mathcal{H}^1,
\dots$ and $\mathcal{F}_0, \mathcal{F}_1, \dots$, respectively. These fit into
short exact sequences $0 \rightarrow X^{i + 1} \rightarrow F_i^X \rightarrow
X^i \rightarrow 0$ for any $i \geq 0$. We say that $X^i$ {\em represents} $X$ in the
$i$-th syzygy.

Our construction will be done
by induction on the number of generating hyperplanes. For this, we will
assume that we already know how to construct such a resolution for $\mH$
which is generated by $t$ hyperplanes $H_1, \dots, H_t$.
Now, if we add another hyperplane $H$, then the intersections $\{H \cap X
\mid X \in \mH \}$ induce a hyperplane arrangement $\mH'$ in $H$
which is generated by at most $t$ hyperplanes. So, our induction assumption
applies to this arrangement as well. The induction step then will be to
show that minimal resolutions of $\mH$ and $\mH'$ can be constructed by
splicing them together in a natural way. We will show the following theorem.

\begin{theorem}\label{res1}
Let $\mathcal{H}$ be an essential hyperplane arrangement of rank $r$. Then there
exists a minimal resolution $0 \rightarrow \mathcal{F}_{r - 1} \rightarrow \cdots
\rightarrow \mathcal{F}_0 \rightarrow \mathcal{H} \rightarrow 0$ by coordinate
arrangements with
\begin{equation*}
F^Y_i = \bigoplus_{\substack{X \leq Y\\ \dim X = i + 1}} F^X_i
\end{equation*}
for every $F^Y_i \in \mathcal{F}_i$. The differential is given by
the sum of natural isomorphisms
$F^X_i \cong X^i$ for any $X$ with $\dim X = i + 1$.
Moreover, for any $X \in \mathcal{H}$, we have $\dim X^k = \beta^{k, X}$
for any $X \in \mathcal{H}$ and any $k \geq 0$. In particular, we get
$\dim X^i = \dim F^X_i = \beta^X$ for any $X \in \mH$ with $\dim X = i + 1$.
\end{theorem}

As an immediate corollary of Theorem \ref{res1} we obtain the following result
for the graded Betti numbers of $E_\mH$.

\begin{theorem}\label{hyperplanebettinumbers}
Let $\mH$ be an essential hyperplane arrangement and $E_\mH$ a reflexive model.
Then the graded Betti number $\beta_i(\uc^X)$ of $E_\mH$ is zero unless
$\dim X = i + 1$. If $\dim X = i + 1$, then $\beta_i(\uc^X)$
coincides with the beta invariant $\beta^X$ of $X$.
\end{theorem}

The rest of this subsection is devoted to the proof of Theorem \ref{res1}.
First, we introduce some more notation.

\begin{definition}
Let $\mathcal{H}$ be a hyperplane arrangement.
For $k \geq 0$ we denote $\mathcal{H}_k = \{X \in \mathcal{H} \mid \dim X = k\}$.
\end{definition}

The following lemma states that
the first step of the resolution can be verified directly.

\begin{lemma}\label{resolutionstart}
Set $\mathcal{F}_0 := \bigoplus_{X \in \mH_1} X$ and $F^Y_0 = \bigoplus_{Y \geq
X \in \mH_1} X$. Then the natural homomorphism of vector space arrangements
$\mathcal{F}_0 \overset{\phi_0}{\rightarrow} \mH$ is surjective and minimal,
i.e. $\phi_0$ does
not factorize through another surjection $\mathcal{F}' \twoheadrightarrow \mathcal{H}$,
with $\mathcal{F}'$ a coordinate arrangement of strictly lower dimension than
$\mathcal{F}_0$.
\end{lemma}

\begin{proof}
The minimality follows from the fact that $\mathcal{H}$ is essential and $\mH_1$
represents the set of its minimal nontrivial vector spaces. Therefore a coordinate vector
space arrangement $\mathcal{F}'$ surjecting onto $\mathcal{H}$ must at least have dimension
$\vert \mH_1 \vert$. It follows from an easy induction on $\dim V$ and $t$
that the restriction of $\phi_0$ to $F^X_0$ is onto $X$ for every $X \in \mathcal{H}$.
\end{proof}

We start now our induction.
Because we only consider essential arrangements, we may begin
with $t = r$ and by assuming without loss  of generality that $H_1, \dots, H_r$
are linearly independent and therefore generate an essential coordinate arrangement
in $V$. It follows that $\phi_0$ from Lemma \ref{resolutionstart}
is injective and therefore we get a minimal resolution of coordinate arrangements
$0 \rightarrow \mathcal{F}_0 \overset{\phi_0}{\rightarrow} \mH \rightarrow 0$
for which the assertions of Theorem \ref{res1} trivially hold.

Now, for $t \geq r$ we have constructed by our induction assumption a minimal
resolution of $\mH$ and $\mH'$ by coordinate arrangements. Denote $\tilde{\mH}$
the new hyperplane arrangement generated by $\mH$ and $\mH'$. We want to combine
the resolutions of $\mH$ and $\mH'$ to a resolution of $\tilde{\mH}$.
We introduce some notation:
\begin{enumerate}
\item We denote $\cdots \mathcal{F}_k \overset{\phi_k}{\rightarrow} \mathcal{F}_{k - 1}$
and $\cdots \mathcal{F}'_k \overset{\phi'_k}{\rightarrow} \mathcal{F}'_{k - 1}$ the
resolutions of $\mH$ and $\mH'$, respectively.
\item For every $k > 0$ set $\mathcal{A}_k := \mH'_k \setminus \mH_k$.
\item For $X \in \mH_k$ denote $X_H = \begin{cases} X \cap H \in \mathcal{A}_{k - 1} &
\text{ if } X \nsubseteq H\\ 0 & \text{ else.}\end{cases}$
\end{enumerate}
In particular, we have $\mathcal{A}_{k - 1} = \{X_H \mid X \in \mH_k\}$.
Our first step in the resolution then will be given by
\begin{equation*}
\tilde{\mathcal{F}}_0 := \mathcal{F}_0 \oplus \bigoplus_{X \in \mathcal{A}_1} X
\xrightarrow{\tilde{\phi}_0} V
\end{equation*}
where $\tilde{\phi}_0$ is the sum of $\phi_0$ and the restriction of $\phi'_0$ to
${\bigoplus_{X \in \mathcal{A}_1} X}$. The free arrangement is given by
$\tilde{F}^Y_0 = \bigoplus_{X \leq Y, X \in \tilde{\mH}_1} X$.
The map $\tilde{\phi}_0$ is surjective by Lemma
\ref{resolutionstart}. We get $\dim \kernel \tilde{\phi}_0 = \dim \kernel \phi_0 +
\vert \mathcal{A}_1 \vert$.
Now consider for every $X \in \mH_1$ the following commutative exact diagram:
\begin{equation}\label{augdiagram}
\xymatrix{
0 \ar[r] & X^1 \ar[r] \ar@{^{(}->}[d] & F_0^X \ar[r]
\ar@{^{(}->}[d] & X \ar[r] \ar@{=}[d] & 0 \\
0 \ar[r] & \tilde{X}^1 \ar[r] & F_0^X
\oplus X_H \ar[r] & X \ar[r] & 0.
}
\end{equation}
Denote $\psi_1$ the sum of all inclusions $\tilde{X}^1 \hookrightarrow \tilde{F}_0$
with $X \in \mH_1$,
then we can read off from this diagram the next step in the resolution:
\begin{equation*}
\tilde{F}_1 := \bigoplus_{X \in \mH_2} \tilde{X}^1 \oplus \bigoplus_{X \in \mathcal{A}_2} X^1
\xrightarrow{\tilde{\phi}_1} \tilde{F}_0.
\end{equation*}
Here, $\tilde{\phi}_1$ is the sum of $\psi_1$ and of $\phi'_1$ restricted to
$\bigoplus_{X \in \mathcal{A}_2}X^1$. Now by diagram (\ref{augdiagram}) and the
observation that every
$\tilde{X}^1$ with $X_H \neq 0$ projects nontrivially to $X_H$ and trivially to every
other $Y_H \in \mathcal{A}_1$, we conclude that $\rk \psi_1 = \rk \phi_1 +
\vert \mathcal{A}_1 \vert = \dim \kernel \tilde{\phi}_0$. It follows that
$\dim \ker \tilde{\phi}_1 = \dim \ker \phi_1 + \sum_{X \in \mathcal{A}_2} \dim X^1$.
To conclude, we observe that $\dim \kernel \tilde{\phi}_1 - \dim \kernel \phi_1 =
\sum_{X \in \mathcal{A}_2} \dim X^1$ and $\rk \tilde{\phi}_1 - \rk \phi_1 =
\sum_{X \in \mathcal{A}_1} \dim X^0$.

Now we assume inductively that we have constructed $\tilde{\phi}_i : \mathcal{F}_k
\rightarrow \mathcal{F}_{k - 1}$ and that the inclusion $\mathcal{F}_k \hookrightarrow
\bigoplus_{X \in \mH_k} \tilde{X}^k$ induces an isomorphism
$\kernel \psi_k \cong \kernel \phi_k$,
where $\psi_k$ denotes the restriction of $\tilde{\phi}$ to
$\bigoplus_{X \in \mH_{k + 1}} \tilde{X}^k$. Moreover, assume that this isomorphism
induces a natural isomorphism
$X^{k + 1} \cong \kernel \psi_k\vert_{\tilde{X}^k}$ for every $X \in \mH_{k + 2}$.
Then we have the following commutative exact diagram for every $X \in \mH_{k + 2}$:
\begin{equation*}
\xymatrix{
0 \ar[r] & X^{k + 1} \ar[r] \ar@{^{(}->}[d] & \bigoplus_{\substack{Y \in \mathcal{H}_{k + 1}
\\ Y \leq X}} \tilde{Y}^k  \ar[r] \ar@{^{(}->}[d] & \tilde{X}^k \ar[r] \ar@{=}[d] & 0 \\
0 \ar[r] & \tilde{X}^{k + 1} \ar[r] & \bigoplus_{\substack{Y \in \mathcal{H}_{k + 1}
\\ Y \leq X}} \tilde{Y}^k \oplus X_H^k \ar[r] & \tilde{X}^k \ar[r] & 0.
}
\end{equation*}
Denoting $\psi_{k + 1}$ the sum of all inclusions $\tilde{X}^{k + 1} \hookrightarrow
\tilde{F}_k$ we set
\begin{equation*}
\tilde{F}_{k + 1} := \bigoplus_{X \in \mH_{k + 2}} \tilde{X}^{k + 1} \oplus
\bigoplus_{X \in \mathcal{A}_{k + 1}} X^{k + 1}
\xrightarrow{\tilde{\phi}_{k + 1}} \tilde{F}_k.
\end{equation*}
with $\tilde{\phi}_{k + 1}$ the sum of $\psi_{k + 1}$ and $\phi'_{k + 1}$ restricted
to $\bigoplus_{X \in \mathcal{A}_{k + 1}} X^k$.
By observing that every $X \in \mH_{k + 1}$ with $X_H \neq 0$ the vector space
$\tilde{X}^{k + 1}$ projects nontrivially to $X_H^k$ and trivially to every other
$Y_H^k$ with $Y_H \in \mathcal{A}_{k + 1}$, we conclude that $\rk \psi_{k + 1}
= \rk \phi_{k + 1} + \sum_{X \in \mathcal{A}_{k + 1}}$. We conclude that
$\dim \kernel \tilde{\phi}_{k + 1} - \dim \kernel \phi_{k + 1} = \sum_{X \in \mathcal{A}_{k +2}}
\dim X^{k + 1}$ and $\rk \tilde{\phi}_{k + 1} - \rk \phi_{k + 1}= \sum_{X \in \mathcal{A}_{k + 1}}
\dim X^k$. In particular, the image of $\tilde{\phi}_{k + 1}$ coincides with the
kernel of $\tilde{\phi_k}$ and by
induction we obtain a minimal resolution of $\tilde{\mH}$ by coordinate arrangements.
This shows the first assertion of Theorem \ref{res1}.

The following proposition proves the assertion on the dimensions of the vector spaces $X^k$.

\begin{proposition}
$\dim X^k = \beta^{k, X}$ for any $X \in \mathcal{H}$. In particular, if $\dim X = i + 1$,
then $\dim X^i = \dim F^X_i = \beta^X$.
\end{proposition}

\begin{proof}
With the notation from Definition \ref{chaindef}
we claim $\dim X^k = (-1)^k \sum_{c \in C_{k, X}} (-1)^{l(c)} \dim c$.

The claim is trivially true for $k = 0$. Now we do induction on $k$ and assume that
$k > 0$. Then we have $\dim X^k = \dim F_{k - 1}^X - \dim X^{k - 1}$ and
$\dim F_{k - 1}^X =
\sum_{Y \leq X, \dim Y = k} \dim Y^{k - 1}$ by Theorem \ref{res1}. Using our induction
assumption, we get
\begin{align*}
\dim X^k & = \sum_{Y \in \p_{k, X}, \dim Y = k} (-1)^{k - 1} \sum_{c \in C_{k - 1, Y}} (-1)^{l(c)}
\dim c - (-1)^{k - 1}\sum_{c \in C_{k - 1, X}} (-1)^{l(c)} \dim c\\
& = (-1)^k \sum_{c \in C_{k, X}} (-1)^{l(c)} \dim c.
\end{align*}
So the claim is proved and the proposition follows with Lemma \ref{betachainlemma}.
\end{proof}

\begin{remark}
In order to relate our results to earlier work on $\Z^n$-graded resolutions
such as \cite{CharalambousTchernev03}, \cite{Tchernev07}, we briefly explain how
we can construct a matrix representation for a hyperperplane module. Assume that
the $H_i$ are given by the orthogonal complement of linear forms $u_1, \dots, u_t$
in $V^*$. Then we get a short exact sequence of $\K$-vector spaces
\begin{equation*}
0 \longrightarrow V \xrightarrow{\ \eta \ } \K^t \xrightarrow{\ \xi \ } \K^{t - r}
\longrightarrow 0,
\end{equation*}
where $\eta$ is represented by a matrix whose $l$-th row is given by $u_l$
and $\xi$ is the corresponding cokernel map. It is not difficult to see that
this sequence can be considered as associated to the following $\Z^n$-graded
exact sequence:
\begin{equation*}
0 \longrightarrow E_\mH \xrightarrow{\ \tilde{\eta}\ } \bigoplus_{i = 1}^t S(-\uc^{H_i})
\xrightarrow{\ \tilde{\xi}\ } S(-\uc^V)^{t - r}.
\end{equation*}
In particular, $\xi$ is the matrix of coefficients of the monomial matrix $\tilde{\xi}$.
This sequence is an example of an {\em Euler-type sequence} (see \cite{PerlingTrautmann}
for some exactness properties). One can now consider the Buchsbaum-Rim complex associated
to $\tilde{\xi}$ as has been done in \cite{CharalambousTchernev03}. As the
maps $\eta$ and $\xi$ are related by Gale duality, it is straightforward to see that
$\xi$ satisfies the uniformity condition of \cite{CharalambousTchernev03}, Def. 4.6
(see also \cite{Tchernev07}, \S 2.4) if and only if the hyperplanes in $\mH$ are in
general position. In that case the Buchsbaum-Rim complex is a minimal resolution of
$E_\mH$. Then by Theorem \ref{hyperplanebettinumbers} we get a nice combinatorial
interpretation of the ranks of the graded parts of the Buchsbaum-Rim complex.
\end{remark}

\subsection{Local cohomology}\label{hyperplanesloccohom}

In this subsection we will determine the local cohomologies $H^i_x E_\mH$ of a
hyperplane module $E_\mH$ with respect to the torus fixed point $x$ in $\A^n_\K$.
Let $\cG$ be the $\gcd$-lattice generated by the degrees of the nonzero graded Bass
numbers of $E_\mH$, as in Definition \ref{adjacentdefinition}.
It suffices to determine the dimensions $\dim (H^i_x E)_\ud$ for every $\ud \in \cG$.

\begin{lemma}\label{partitionlemma}
Let $\ud \in \cG$ and denote $J = \{j \in \on \mid \ud \leq \ud^{H_j}\} \subseteq \on$.
Moreover, denote $\mathcal{H}_\ud$ the hyperplane arrangement generated by the
$\{H_j\}_{j \in J}$. Then for every $k \geq 0$  the set $\{X \in \mathcal{H}_k \mid
\ud \leq \ud^X\}$ equals $\mathcal{H}_\ud \cap \mathcal{H}_k$.
\end{lemma}

\begin{proof}
$\ud^V$ is the unique maximal element in $\cG$, thus $\ud \leq \ud^V$ for every
$\ud \in \cG$. By the construction of subsection \ref{reflexivelcmlattices} we have
$\ud^X = \gcd\{\ud^{H_i} \mid X \subseteq H_i\}$ for every $X \in \mathcal{H}$.
Hence $\ud \leq \uc^X$ for every $X \in \mathcal{H}_\ud$ and thus
$\{X \in \mathcal{H}_k \mid \ud \leq \ud^X\}$ = $\mathcal{H}_\ud \cap
\mathcal{H}_k$.
\end{proof}

According to Proposition \ref{cohomologyshift}, the $i$-th local cohomology
at degree $\ud \in \cG$ is given by $(H^i_x E)_\ud \cong H^{i + 1}(B^{\bullet, \ud})$,
where $B^{\bullet, \ud}$ is a certain complex of vector spaces. With Lemma
\ref{partitionlemma} we can describe this complex explicitly as follows.
Consider the two subarrangements $\mH_\ud$ and
$\mH^\ud$ of $\mH$, where $\mH_\ud$ is generated by  all $H \in \mH$ such that
$\ud \leq \ud^H$ and $\mH^\ud$ is generated by the complementary set
of hyperplanes in $\mH$. In the following lemma we first collect the degenerate
cases where $\mH^\ud$ contains at most one hyperplane.

\begin{lemma}\label{lcohomdegrees}
\begin{enumerate}[(i)]
\item\label{lcohomdegreesii} If $\mH^\ud$ is generated by one hyperplane $H$ then
$H^{i + 1}(B^{\bullet, \ud}) = 0$ for all $i < n$ and $H^{n + 1}(B^{\bullet, \ud})$ is naturally isomorphic to $V / H$.
\item\label{lcohomdegreesiii} If $\mH^\ud$ does not contain any hyperplane then $B^{i, \ud} = 0$
for $i \leq n$ and $H^{n + 1}(B^{\bullet, \ud})$ is naturally isomorphic to $V$.
\end{enumerate}
\end{lemma}

\begin{proof}
The statements of follow by inspection of the minimal resolution by
coordinate arrangements of $\mH$, where for (\ref{lcohomdegreesii})
we remark that the terms $B^{i, \ud}$ for $i \leq n$ coincide with the
minimal resolution of the hyperplane arrangement $\mH$ restricted to $H$.
\end{proof}

For the remaining case denote $\mathcal{B}^\ud_j
:= \mathcal{H}_k \setminus \mathcal{H}^\ud$ for every $k \geq 0$. For every $k \leq n$
set $B^{k, \ud} := \bigoplus_{X \in \mathcal{B}^\ud_{n - k + 1}} X^{n - k}$
and $B^{n + 1, \ud} := E$. Moreover, let $\eta^k$ be the restriction
of $\phi_{n - k}$ to $\bigoplus_{X \in \mathcal{B}_{n - k + 1}} X^{n - k}$.
Then the complex $B^{\bullet, \ud}$ has the following shape:
\begin{equation}\label{loccohomsequence}
0 \longrightarrow B^{n - t + 1, \ud} \xrightarrow{\eta^{n - t + 1}}
\cdots \overset{\eta^1}{\longrightarrow} B^{n, \ud} \overset{\eta^n}{\longrightarrow}
B^{n + 1, \ud} = E \longrightarrow 0.
\end{equation}
Its cohomoloy is described by the following proposition.

\begin{proposition}\label{lcohomprop}
Denote $t$ the length of a minimal resolution of $H^\ud$ and assume that $H^\ud$ is
generated by at least two hyperplanes. Then $H^i(B^{\bullet, \ud}) $ is zero for $i
\neq n - t + 1$ and $H^{n - t + 1}(B^{\bullet, \ud})$ is isomorphic to the $t$-th
syzygy of $V$ with respect to a minimal resolution of $\mH^\ud$. In particular
the complex $B^{\bullet, \ud}$ is exact iff $H^\ud$ is a coordinate arrangement.
\end{proposition}

\begin{proof}
Assume that $\mH^\ud$ is essential and consider a minimal resolution of $\mH^\ud$
by coordinate arrangements
$0 \rightarrow \mathcal{M}_t \rightarrow \cdots \rightarrow \mathcal{M}_0 \rightarrow
\mH^\ud \rightarrow 0$. By enumerating the hyperplanes in $\mH_\ud$: $H_1, \dots, H_r$,
we can augment this resolution as follows. For every $0 \leq k < t$ and $i \geq 0$ we
inductively define sets $\mathcal{S}^0_k = \mH^\uc_k$ and $\mathcal{S}^i_k =
\{X \cap H_i \mid X \in \mathcal{S}^{i - 1}_{k + 1}\}$ for $i > 0$. We denote
$\mathcal{A}^i_k := \mathcal{S}^i_k \setminus \mathcal{S}^{i - 1}_k$ for $i > 0$.
Now for every $0 \leq k < t$ we set $\mathcal{M}^0_k := \mathcal{M}_k$ and
$\mathcal{M}^i_k := \mathcal{M}^{i - 1} \oplus \bigoplus_{X \in \mathcal{A}^i_k} X^{k - 1}$
for $i > 0$. Moreover, we define the differential among the $\mathcal{M}^i_k$
as the restrictions of the corresponding differentials $\phi_k$ of the minimal resolution
of $\mH$. We end up with the complex
\begin{equation*}
0 \longrightarrow \mathcal{M}^r_{t - 1} \longrightarrow \cdots \longrightarrow
\mathcal{M}^r_0 \longrightarrow 0.
\end{equation*}
We can argue now analogous to the proof of Theorem \ref{res1} that our successive
extensions lead to no new homology and the only homologies of this complex are
$H^0(\mathcal{M}^r_\bullet) \cong V$ and $H^{t - 1}(\mathcal{M}^r_\bullet)$
is isomorphic to $\mathcal{M}_t$. Then by passing to cohomological degrees and
by forgetting about the underlying vector space arrangements we see that the
complex $0 \rightarrow \mathcal{M}^r_{t - 1} \rightarrow \cdots \rightarrow
\mathcal{M}^r_0 \rightarrow V \rightarrow 0$ coincides with the complex
(\ref{loccohomsequence}) and our assertions follow.

In the case that $\mH^\ud$ is not essential with center $C$ assume that we are given a minimal
resolution $0 \rightarrow \mathcal{M}_t \rightarrow \cdots \rightarrow \mathcal{M}_0
\rightarrow \mH^\ud / C \rightarrow 0$ and assume we add another hyperplane $H$ to
$\mH^\ud$. In the case that $H \supseteq C$, the center remains the same and we
can apply our above discussion with respect to the vector space $V / C$.
If $H \nsupseteq C$,
then $H$ properly intersects all $X \in \mH^\ud$. In that case it is not difficult to see
that the complex changes to $0 \rightarrow \mathcal{M}_t \rightarrow \cdots \rightarrow
\mathcal{M}_1 \rightarrow \mathcal{M}_0 \oplus H \cap X \rightarrow \mathcal{H} / C'$,
where $C' \cong C / C \cap H$, i.e. only the $0$-th term of the resolution changes and
immediately is compensated.
\end{proof}

For any $\mH^\ud$ with center $C^\ud$, we can associate to it the beta invariant
$\beta^\ud$ which is the beta invariant of the essential arrangement $\mH^\ud / C^\ud$
in the sense of subsection \ref{combinatorics}. Its {\em rank} is given by
$r^\ud = \dim V / C^\ud$. If $\mH^\ud$ does not contain any hyperplane, then we
use the convention that $C^\ud = 0$, hence $r^\ud = \dim V$.
Then Theorem \ref{hyperplanebettinumbers} tells us that $\beta^\ud$ equals the
dimension of the $(r^\ud - 1)$st syzygy of $V / C^\ud$ with respect to a minimal
resolution $\mH^\ud$ by coordinate arrangements.
We can now show:

\begin{theorem}\label{basstheorem}
Let $E = E_\mH$ be  a reflexive model of a hyperplane arrangement $\mH$ in $V$.
Denote $\cG$ the $\gcd$-lattice generated by the degrees of the Bass numbers of
$E_\mH$. For every $\ud \in \cG$ denote $\mH^\ud \subset \mH$ the hyperplane
arrangement generated by those hyperplanes $H \in \mH$ with $\ud \nleq \ud^H$,
$r^\ud$ its rank and $\beta^\ud$ its beta invariant.
Then for any $\ud \in \cG$ and any $\uc \in \Z^n$ adjacent to $\ud$ we have
\begin{equation*}
\dim (H_x^i E_\mH)_\uc =
\begin{cases}
\beta^\ud & \text{ if \ } i = n - r^\ud + 1\\
0 & \text{ else}.
\end{cases}
\end{equation*}
\end{theorem}

\begin{proof}
If $\mH^\ud$ contains at least two hyperplanes, this follows from Propositon
\ref{lcohomprop}. In the case that $\mH$ contains less than two hyperplanes,
this follows from Lemma \ref{lcohomdegrees}, where we remark that, if $\mH^\ud$
is generated by one hyperplane $H$, then $C^\ud = H$ and thus $\beta^\ud = \dim
V / H$.
\end{proof}

\section{Codivisorial resolutions and  maximal Cohen-Macaulay modules}\label{mcmmodules}

In this section we will use the results of the previous sections to construct
examples of $M$-graded MCM modules over rings $\ksm$, where $\sigma$ is a simplicial
cone.
Our strategy will be to start with a certain reflexive $S$-module $\hat{E}$ and
to consider the graded structure of its minimal injective resolution. Now, if $S$
serves as a homogeneous coordinate ring for $\ksm$, we can derive conditions
on $\sigma$ which tell us when $\hat{E}_{(0)}$ can be a MCM
module over $\ksm$. More precisely, denote $x \in U_\sigma$ the torus fixed point
and $\hat{x} \subseteq \A^n_\K$ its preimage under the surjection $\A^n_\K
\twoheadrightarrow U_\sigma$. Then we have $H^i_x \hat{E}_{(0)} =
(H^i_{\hat{x}} \hat{E})_{(0)}$ for any $i \in \Z$ (see Proposition \ref{loccohomprop}).
So, $\hat{E}_{(0)}$ is an
MCM-module iff $(H^i_{\hat{x}} \hat{E})_{(0)} = 0$ for all $i < \dim \ksm$.
If $\hat{E}$ is not free then we cannot expect that $H^i_{\hat{x}} \hat{E}$
will vanish for all such $i$. So the conditions we want to derive will tell us
when $\sigma$ corresponds to an embedding $0 \rightarrow M \rightarrow \Z^n$
such that the nonvanishing degrees of $H^i_{\hat{x}} \hat{E}$ do not intersect
$M$. Note that these kind of conditions will probably not lead to a classification
of MCM modules over a fixed $\ksm$ with respect to, say, certain combinatorial
invariants. As for the rank one case, any such direct approach would lead to
rather complicated arithmetic conditions (see \cite{perling11}), whose characterization
is beyond the scope of this paper.

\subsection{Homogeneous coordinates and local cohomology}\label{homogeneousloccohom}

Let $E$ be an $M$-graded $\ksm$-module ($\sigma$ is not necessarily simplicial)
and $F$ an $\Z^n$-graded $S$-module with
$F_{(0)} \cong E$. Then any $\Z^n$-graded injective $I^\bullet$ resolution of $F$
induces a codivisorial resolution $I^\bullet_{(0)}$ of $E$. Of course, the shape
of this resolution strongly depends on the choice of $F$. A standard choice
is given by $F = \Gamma_* E$, where $\Gamma_* E$ denotes the graded tensor product
$E \otimes_\ksm S$. The $S$-module $\Gamma_* E$ has a natural $\Z^n$-grading which
is given by  $\Gamma_* E \cong \bigoplus_{\uc \in \Z^n} \big(E \otimes_\ksm S(\uc)_{(0)}\big)_0$. In particular, we have a right exact functor
\begin{equation*}
\Gamma_* : \mksmMod \longrightarrow \zsMod.
\end{equation*}
By observing that $(\Gamma_* E)_0 = E$ for any $E$ it follows that the functor
$( - )_{(0)}$ is essentially surjective (see also \cite{Mustata1}, and \cite{CoxBat}
for the $\Z^n$-graded case).

In general, the module $\Gamma_* E$ might not be the best choice for a representative
of $E$. For instance, even in good cases such as $E$ reflexive, $\Gamma_* E$ usually
has torsion. However, as in this section we are interested exclusively in reflexive
modules, we can give an alternative construction in this case.

\begin{definition}
Let $E$ be a finitely generated $M$-graded reflexive $\ksm$-module given by a vector
space $\mathbf{E}$ and filtrations $E^k(i)$. Then we denote $\hat{E}$ the reflexive
$\Z^n$-graded $S$-module associated to the same data.
\end{definition}

In other words, as the fan of $\A^n_\K$ has the same number of rays as $\sigma_M$,
we obtain $\hat{E}$ by simply reinterpreting the filtrations associated $E$. To see
the effect, compare the graded pieces of $E$ and $\hat{E}$, respectively:
\begin{equation*}
E_m = \bigcap_{k \in \on} E^k\big(l_k(m)\big) \quad \text{ and } \quad
\hat{E}_\uc = \bigcap_{k \in \on} E^k(c_k)
\end{equation*}
for every $m \in M$ and every $\uc \in \Z^n$. So, as not every intersection of the
filtrations $E^k(i)$ must be realized as a graded component of $E$, we can consider
$\hat{E}$ as the completion of $E$ with respect to intersections among the vector
spaces in the filtrations.
The following proposition shows that this construction also is functorially well-behaved.

\begin{proposition}
With $E$ and $\hat{E}$ as above we get:
\begin{enumerate}[(i)]
\item\label{liftingpropertyi} Passing from $E$ to $\hat{E}$ defines a fully faithful
functor from the category of $M$-graded, finitely generated reflexive $\ksm$-modules
to the category of $\Z^n$-graded, finitely generated reflexive $S$-modules.
\item\label{liftingpropertyii} $\hat{E}_{(0)} \cong E$.
\item\label{liftingpropertyiii} In particular, the pair of functors $( - )_{(0)}$, $\hat{-}$
induces an equivalence of categories.
\end{enumerate}
\end{proposition}

\begin{proof}
For (\ref{liftingpropertyi}) we remark that
vector space homomorphisms which are compatible with filtrations are also
compatible with intersections of filtrations.

For (\ref{liftingpropertyii}) observe that $E_m = \bigcap_{k \in \on} E^k\big(l_k(m)\big)
= \hat{E}_{L(m)}$ for every $m \in M$ (here, $L$ denotes the inclusion of $M$ into $\Z^n$).

Then (\ref{liftingpropertyiii}) follows directly from (\ref{liftingpropertyi}) and
(\ref{liftingpropertyii}).
\end{proof}

Once chosen a suitable $S$-module $F$ such that $E \cong F_0$, the following
proposition tells us how we can obtain the local cohomologies for $E$ by
computing those of $F$.

\begin{proposition}\label{loccohomprop}
Let $E$ be in $\mksmMod$ and $F$ in $\zsMod$ such that $E \cong F_{(0)}$. Moreover,
let $V \subseteq U_\sigma$ be a torus invariant closed subvariety and $\hat{V}$ the
preimage of $V$ in $\A^n_\K$. Then $H^i_V E \cong (H^i_{\hat{V}} F)_{(0)}$ for every
$i \in \Z$.
\end{proposition}

\begin{proof}
Assume that we have an $\Z^n$-graded injective resolution $0 \rightarrow F \rightarrow
I^\bullet$. By taking degree zero, we obtain a $M$-graded codivisorial resolution
$0 \rightarrow E \rightarrow I^\bullet_{(0)}$. We obtain $H^i_{\hat{V}} F$ as the
$i$-th cohomology of the complex $\Gamma_{\hat{V}} I^\bullet$ and, by Lemma
\ref{acycprop} (\ref{acycpropii}), we obtain $H^i_V E$ as the
$i$-th cohomology of the complex $\Gamma_V I^\bullet_{(0)}$. Using Lemma
\ref{acycprop} (\ref{acycpropii}), we see that the complexes
$(\Gamma_{\hat{V}} I^\bullet)_{(0)}$ and $\Gamma_V I^\bullet_{(0)}$ coincide.
\end{proof}

\subsection{Vector space arrangements with trivial intersections}\label{intersectionfreemcms}

Assume that $\sigma$ is simplicial but not smooth and $E$ an indecomposable reflexive
$\ksm$-module such that
$\hat{E}$ corresponds to a reflexive $S$-module of the type as considered in
Examples \ref{einfach} and \ref{einfachbass}, respectively. Using Proposition
\ref{loccohomprop}, we want to analyze the $\Z^n$-degrees of the local cohomology
modules $H^i_x \hat{E}$, where $x$ is the torus fixed point in $U_\sigma$,
with respect to the embedding of $M$ in $\Z^n$ in order to obtain conditions
for the vanishing of $(H^i_x \hat{E})_{(0)}$ for $i < d$.

Using the notation of Example \ref{einfachbass}, we call, as in Definition
\ref{adjacentdefinition}, an element $\uc \in \Z^n$
{\em adjacent} to $\ud_P$, if $\uc \leq \ud_P$ and $\uc \not \leq \ud_Q$ for all
$Q \nsubseteq P$. Then, for any $\uc$ adjacent to some $\ud_P$, we have
$(H^i_x \hat{E})_\uc = (H^i_x \hat{E})_{\ud_P}$. We have seen in Example
\ref{einfachbass} that $\dim (H^{n - 1}_x \hat{E})_{\ud_P} = \sum_{i \notin P}
\dim V_i - \dim \sum_{i \notin P} V_i$ for every $P \subseteq \{1, \dots, t\}$.
In other words, $(H^{n - 1}_x \hat{E})_{\ud_P}$
vanishes iff the vector spaces $V_i$ indexed by the complement of $P$ form a
linearly independent system of subvector spaces of $V$. We want to simplify
our discussion by considering only cases where for any nonvanishing
$(H^{n - 1}_x \hat{E})_{\ud_P}$ there are only finitely many $\uc \in \Z^n$
adjacent to $\ud_P$.

\begin{lemma}\label{finiteadjacentlemma}
Assume that to any nonvanishing $b^{n - 1}(\ud_P)$ there are only finitely many
$\uc \in \Z^n$ adjacent to $\ud_P$. Then $t = n$ and for any $i \in P$ the
vector spaces $\{V_j\}_{j \neq i}$ form a linearly independent system of subvector spaces.
\end{lemma}

\begin{proof}
The assumption implies that $H^{n -1}_x \hat{E}$ is finite and therefore finitely
generated. By \cite{EGAII}, \S VIII, Cor. 2.3, for $H^{n - 1}_x \hat{E}$ being finitely
generated it is necessary and sufficient that every localization $\hat{E}_{x_i}$
for $i \in \on$ is a free $S_{x_i}$-module, which implies the assertion.
\end{proof}

\begin{corollary}\label{indep}
Under the assumptions of Lemma \ref{finiteadjacentlemma}, $(H^{n - 1}_x
\hat{E})_{\ud_\emptyset}$ is the only nonvanishing graded Bass number in
cohomological degree $n - 1$.
\end{corollary}

\begin{corollary}\label{equidim}
The vector spaces $V_i$ all have the same dimension $k$ and $\dim V = (n - 1) k$.
\end{corollary}

\begin{proof}
First observe that $V_i \subseteq \sum_{j \neq i} V_j$ for every $i \in \on$, as
otherwise we could
split $V_i = V_i' + V_i''$, where $V_i' = V_i \cap \sum_{j \neq i} V_j$ and
this way obtain a splitting of vector space arrangements $V \cong \sum_{j \neq i} V_j
\oplus V_i''$. By Proposition \ref{filtaux} this would contradict the indecomposability
of $\hat{E}$. With this observation, the assertion follows immediately.
\end{proof}

Next we will see that indecomposability induces even stronger conditions.

\begin{proposition}\label{indec}
$\dim V_i = 1$ for every $i \in \on$. In particular, an indecomposable $M$-graded
MCM-module satisfying the conditions of Lemma \ref{finiteadjacentlemma} has
rank $n - 1$.
\end{proposition}

\begin{proof}
By Corollaries \ref{indep} and \ref{equidim}, any choice of basis for $V_1, \dots,
V_{n - 1}$ yields a basis for $V$. Now assume without loss generality that $V_n$
is in general position, i.e. the projection $\pi_i : V_n \rightarrow V_i$ is surjective
for $1 \leq i < n$. So, a choice of basis $v_1, \dots, v_k$ of $V_n$ induces a
basis $\pi_i(v_1), \dots, \pi_k$ of $V_i$ for every $1 \leq i < n$. Hence, for
every $1 \leq j \leq k$, we get $W_j := \sum_{i = 1}^{n - 1} \K \pi_i(v_j)$ with
$\dim W_j = n - 1$ and $v_j \in W_j$. Moreover, every $W_j$ is a filtered vector
space with filtrations given by $\pi_i(v_j)$ for $1 \leq i < n$ and $v_j$, respectively,
such that the decomposition $V \cong \bigoplus_{j = 1}^k W_j$ is a decomposition of
filtered vector spaces. By indecomposability of $E$ it follows that $k = 1$.
\end{proof}

The degrees adjacent to $\ud_\emptyset$ form a cuboid lattice polytope
given by $C = \{\uc \in \Z^n \mid \uc \leq \ud_\emptyset$ and $\uc \nleq \ud_{\{i\}}$
for every $i \in \on\}$. We obtain the following general criterion.

\begin{theorem}\label{cuboidtheorem}
Under the assumptions of Lemma \ref{finiteadjacentlemma}, $E$ is an MCM-module
over $\ksm$ if and only if $C$ is contained in $\Z^n \setminus M$.
\end{theorem}

Theorem \ref{cuboidtheorem} remains somewhat vague as for a given
$\sigma$, we leave open the problem of classifying the admissible cuboid regions
$C$. The following theorem gives a precise statement for the case where
$C$ consists of only one point. Let
$C = \{\ud_\emptyset =: \ud\}$. We consider equivariant isomorphism class of
reflexive $\ksm$-modules up to degree-shift by elements of $M$. In terms of filtrations
this means that we consider
filtrations $E^k(i)$ up to a simultaneous shift $E^k\big(i + l_k(m)\big)$ for
every $k \in \on$ by some $m \in M$. The following theorem classifies all
isomorphism types up to degree-shift in terms of the divisor class
group $A_{n - 1}(U_\sigma)$.

\begin{theorem}\label{oneelement}
Let $\sigma$ be a simplicial cone. Then there are, up degree-shift in $M$,
precisely $| A_{n - 1}(U_\sigma) | - 1$ isomorphism classes of indecomposable
MCM modules satisfying the assumptions of Lemma \ref{finiteadjacentlemma} such
that $C$ consists of one element.
\end{theorem}

\begin{proof}
With notation as above, the possible $\ud$ are classified
by the cokernel of the short exact sequence
$0 \rightarrow M \xrightarrow{ L } \Z^n \rightarrow A_{n - 1}(U_\sigma) \rightarrow 0$.
The module $E$ is MCM iff $\ud$ represents a nonzero element in
$A_{n - 1}(U_\sigma)$.

By Proposition \ref{indec}, we can parameterize all modules $E$ satisfying above
conditions by configurations of $n$ points in general position in $\mathbb{P} V
\cong \mathbb{P}^{n - 2}$
up to the action of $\operatorname{GL}_\K(V)$, which leaves us with precisely one
isomorphism class.
\end{proof}

\begin{remark}
Note that our conditions imply that $\dim U_\sigma > 2$.
It is well-known that equivariant reflexive modules over affine toric surfaces
always split into a direct sum of modules of rank one.
\end{remark}

\begin{example}
Consider the cone $\sigma$ in $\Z^3$ spanned by primitive vectors $l_1 = (1, 0, 1)$,
$l_2 = (0, 1, 1)$, and $l_3 = (-1, -1, 1)$. Denote $\hat{E}$ a rank two reflexive $S$-module
given by three lines $V_1, V_2, V_2$ in general position in $\mathbb{E} \cong \K^2$.
We show that, up to degree-shift in $M$, there exist precisely five indecomposable
equivariant isomorphism classes of MCM modules of rank two associated to this data.

Assume that $\ud = (b_1, b_2, b_3)$ and $\ud_1 = (a_1, b_2, b_3)$,
$\ud_2 = (b_1, a_2, b_3)$, $\ud_3 = (b_1, b_2, a_3)$ with $a_k < b_k$ for $1 \leq k \leq 3$.
Its local cohomology is given by the cohomology of the following complex which
is concentrated in cohomological degrees $2$ and $3$:
\begin{equation*}
0 \longrightarrow \dual{S}(-\ud) \longrightarrow
\dual{S}(-\ud_1) \oplus \dual{S}(-\ud_2) \oplus
\dual{S}(-\ud_3) \longrightarrow 0.
\end{equation*}
Up to degree-shift, there are two isomorphism classes of MCM modules with $C = \{\ud\}$
which are classified by
\begin{equation*}
0 \longrightarrow M \xrightarrow{\ L\ } \Z^3 \longrightarrow \Z / 3 \Z \longrightarrow 0.
\end{equation*}
For instance, we can choose $\ud \in \{(1, 1, 0), (2, 2, 0)\}$.

A small computation shows that $C$ cannot have more than two elements. For the
case that $C$ has two elements, we can find three more isomorphism classes,
where $C$ can be represented by pairs $\{(1, 1, 0), (1, 1, 0) - e_k\}$, where for
$1 \leq k \leq 3$, $e_k$ denote the standard basis vectors of $\Z^3$ . Note that,
as argued in the proof of Theorem \ref{oneelement}, for any
given $C$, there are no nontrivial equivariant moduli associated to $E$, we have
precisely one isomorphism class.
\end{example}

\begin{remark}
Assume that we have an MCM-module $E$ as above with $C = \{\ud = (d_1, \dots, d_n)\}$.
Then it follows from the general theory that its canonical dual $F := \Hom(E, \omega)$
(here, $\omega$ denotes the canonical module of $\ksm$, which is isomorphic to
$S(\one)_{(0)}$) is MCM as well.
However, it is interesting to verify this fact explicitly using our framework.
The filtrations associated to $E$ are given for $k \in \on$ by
\begin{equation*}
E^k(i) =
\begin{cases}
0 & \text{ for } i < d_k\\
V_k & \text{ for } i = d_k\\
V & \text{ for } i > d_k.
\end{cases}
\end{equation*}
Now, for its MCM-dual $F := \Hom(E, \omega)$, we obtain the
filtrations
\begin{equation*}
F^k(i) =
\begin{cases}
0 & \text{ for } i < -d_k - 1\\
H_k & \text{ for } i = -d_k - 1\\
V^* & \text{ for } i \geq -d_k,
\end{cases}
\end{equation*}
where $V^* = \Hom(V, \K)$ and $H_k$ denotes the orthogonal complement of $V_k$ in
$V^*$ with respect to the canonical pairing $V^* \times V \rightarrow \K$.
The hyperplanes $H_1, \dots, H_n$ form a hyperplane arrangement $\mH$ in general
position in $V^*$. The minimal resolution of $\hat{F}$ is given by
\begin{equation*}
0 \longrightarrow S(-\uc^V) \longrightarrow \cdots \longrightarrow
\bigoplus_{X \in \mH_k} S(-\uc^X) \longrightarrow \cdots \longrightarrow
\bigoplus_{X \in \mH_1} S(-\uc^X) \longrightarrow \hat{F} \longrightarrow 0,
\end{equation*}
where $\uc^X = (c_1, \dots, c_n)$ with $c_k = -d_k - 1$ whenever $X \subseteq H_k$
and $c_k = -d_k$ else. In particular, $\uc^V = -\ud$.
The local cohomology $H^i_x \hat{F}$ then is the cohomology of the following
complex which is concentrated in cohomological degrees $2$ to $n$:
\begin{equation*}
0 \longrightarrow \dual{S}(-\underline{f}^V) \longrightarrow \cdots \longrightarrow
\bigoplus_{X \in \mH_k} \dual{S}(-\underline{f}^X) \longrightarrow \cdots \longrightarrow
\bigoplus_{X \in \mH_1} \dual{S}(-\underline{f}^X) \longrightarrow 0,
\end{equation*}
where $\underline{f}^X = \uc^X + \one$ for every $X \in \mH$.
By Theorem \ref{basstheorem}, we have $H^i_x \hat{E} = 0$ for $2 < i < n$
and $(H^2_x \hat{E})_{\underline{f}^V} \cong \K$ for $\uc \in C$ and
$(H^2_x \hat{E})_{\uc} = 0$, for $\uc \neq \underline{f}^V$. We can conclude
that $F$ is MCM iff $\ud + \one \in \Z^n \setminus M$.
\end{remark}

\end{document}